\documentclass[11pt]{amsart}
\usepackage{amsmath,amssymb,mathtools}
\usepackage[a4paper,margin=1.1in]{geometry}
\usepackage{url}
\usepackage[colorlinks=true,allcolors=blue]{hyperref}

\newtheorem{theorem}{Theorem}[section]
\newtheorem{proposition}[theorem]{Proposition}
\newtheorem{lemma}[theorem]{Lemma}
\newtheorem{corollary}[theorem]{Corollary}
\theoremstyle{definition}

\newtheorem{example}[theorem]{Example}
\newtheorem{remark}[theorem]{Remark}

\newcommand{\N}{\mathcal N}
\newcommand{\res}[1]{{#1}^{\N}}
\newcommand{\sn}{\mathrel{\trianglelefteq\trianglelefteq}}
\newcommand{\W}{\mathrm W}
\newcommand{\ellW}{\ell_{\mathrm W}}
\newcommand{\bN}{b_{\N}}
\newcommand{\Def}{\operatorname{def}}

\newcommand{\Soc}{\operatorname{Soc}}
\newcommand{\Aut}{\operatorname{Aut}}
\newcommand{\Out}{\operatorname{Out}}
\newcommand{\Fit}{\mathrm F}
\newcommand{\Core}{\operatorname{Core}}
\newcommand{\Rad}{\mathrm R}
\newcommand{\Layer}{\mathrm E}
\newcommand{\wl}{\operatorname{wl}}
\newcommand{\Inn}{\operatorname{Inn}}
\newcommand{\GL}{\operatorname{GL}}
\newcommand{\SL}{\operatorname{SL}}
\newcommand{\AGL}{\operatorname{AGL}}
\newcommand{\Ind}{\operatorname{Ind}}
\newcommand{\FF}{\mathbb F}

\title[Wielandt subgroups of nilpotent residuals]%
 {Wielandt subgroups of nilpotent residuals and Fitting-height bounds}

\author{L\"u Gong}
\address{School of Mathematics and Statistics, Nantong University,
Nantong, Jiangsu 226019, China}
\email{lieningzai1917@126.com}

\author{Yan Yan}
\address{School of Mathematics and Statistics, Nantong University,
Nantong, Jiangsu 226019, China}
\email{yanyan320124@163.com}

\author{Enric Cosme Ll\'opez}
\address{Departament de Matem\`atiques, Universitat de Val\`encia,
C/ Dr.~Moliner, 50, 46100 Burjassot, Val\`encia, Spain; and
School of Mathematics and Statistics, Nantong University,
Nantong, Jiangsu 226019, China}
\email{enric.cosme@uv.es}


\begin{document}

\begin{abstract}
For a finite group $G$, let $\mathrm W(G)$ be the intersection of
the normalizers of the nilpotent residuals of the subnormal subgroups
of $G$, and let $\ell_{\mathrm W}(G)$ be the length of the associated
upper series. Thus $\mathrm W(G)=G$ exactly when all these residuals
are normal in $G$.

We prove that every finite soluble group satisfies
$\ell_{\mathrm W}(G)\le h(G)\le4\ell_{\mathrm W}(G)$, where $h(G)$ is
the Fitting height, and that $h(G)\le3\ell_{\mathrm W}(G)$ if $|G|$ is
odd or the Sylow $2$-subgroups of $G$ are abelian. Both constants are
attained at every length, the constant $4$ by $\{2,3\}$-groups.
Moreover, every pair $(n,h)$ with $n\le h\le4n$ occurs as
$(\ell_{\mathrm W}(G),h(G))$.

The constant $4$ comes from primitive groups. A soluble primitive group
$V\rtimes M$ satisfies $\mathrm W(V\rtimes M)=V\rtimes M$ if and only
if every abelian subgroup of $\mathrm F(M)$ is cyclic. It then has
Fitting height four if and only if $\mathrm F(M)=C\times Q_8$, with $C$
cyclic, and $M$ induces the full group $\mathrm{Out}(Q_8)\cong S_3$. We
also determine when the base group of a soluble regular wreath product
is the first term of the series.
\end{abstract}

\maketitle

\section{Introduction}
All groups in this paper are finite. We write $H\sn G$ for
subnormality; the remaining notation is collected in
Section~\ref{sec:prelim}.

\subsection{Background: the Wielandt series}
The \emph{Wielandt subgroup}
\[
\omega(G)=\bigcap_{H\sn G}N_G(H)
\]
was introduced by Wielandt \cite{Wielandt}, who proved that it
contains every minimal normal subgroup of $G$; in particular,
$\omega(G)\ne1$ whenever $G\ne1$. Consequently the upper Wielandt
series
\[
\omega_0(G)=1,\qquad
\omega_{i+1}(G)/\omega_i(G)=\omega(G/\omega_i(G))
\]
reaches $G$, and the least index $i$ with $\omega_i(G)=G$ is the
\emph{Wielandt length} $\wl(G)$. The groups of Wielandt
length one are the $T$-groups, in which normality is transitive;
Gasch\"utz \cite{Gaschutz} proved that soluble $T$-groups are
metabelian, so they have Fitting height at most two.

If $H\sn G$, then
$H\trianglelefteq H\omega_1(G)\trianglelefteq\cdots\trianglelefteq
H\omega_{\wl(G)}(G)=G$, so the subnormal defect of $H$ is at most
$\wl(G)$. For soluble groups, Camina \cite{Camina}
showed that the derived length and the Fitting height are bounded in
terms of the Wielandt length, and Bryce and Cossey
\cite{BryceCossey} found the best possible bounds: if $G$ is soluble
and $\wl(G)=n$, then
\[
h(G)\le n+1,
\]
and the derived length of $G$ is at most $\lceil5n/3\rceil$ (see
\cite[Theorems~1 and~2]{CasoloFiniteLength}, where these bounds are
stated and extended to infinite soluble groups). Casolo studied how
the Wielandt length is controlled by the Fitting height together with
the subnormal defects \cite{CasoloSeries}.

\subsection{The residual-normalizer subgroup}
The Wielandt subgroup is sensitive to the fine structure of nilpotent
groups. Every subgroup of a nilpotent group $G$ is subnormal, so
$\omega(G)$ is then the norm $\bigcap_{H\le G}N_G(H)$ introduced by
Baer \cite{Baer}, which is contained in $Z_2(G)$ by a theorem of
Schenkman \cite{Schenkman} and is in general a proper subgroup. For
instance, $\omega(D_8)=Z(D_8)$ and $\wl(D_8)=2$. Since we
are interested in the nonnilpotent part of a group, we replace each
subnormal subgroup by its nilpotent residual. Let $\N$ be the class of
nilpotent groups, and put
\[
H^{\N}=\bigcap\{K\unlhd H:H/K\text{ is nilpotent}\}=\gamma_\infty(H),
\qquad \W(G)=\bigcap_{H\sn G}N_G(H^{\N}).
\]
The \emph{upper $\W$-series} is
\[
\W_0(G)=1,\qquad \W_{i+1}(G)/\W_i(G)=\W(G/\W_i(G)),
\]
and its length is denoted $\ellW(G)$, with $\ellW(1)=0$. We call $X$
a \emph{$\W$-group} if $\W(X)=X$, that is, if the nilpotent residual
of every subnormal subgroup of $X$ is normal in $X$.

Since $H^{\N}$ is characteristic in $H$, we have $\omega(G)\le\W(G)$,
and hence $\ellW(G)\le\wl(G)$
(Theorem~\ref{thm:general-height}). Every nilpotent group is a
$\W$-group, so $\ellW(D_8)=1$. More generally, $\W(G)$ contains the
generalized Fitting subgroup $\Fit^*(G)$ (Lemma~\ref{lem:Fstar}). On
the other hand, $\W(G)$ can be a proper subgroup: in
$G=S_3\wr C_2=(S_3\times S_3)\rtimes C_2$ the subgroup $S_3\times1$ is
subnormal with residual $C_3\times1$, which is not normalized by the
involution interchanging the factors; in fact $|G:\W(G)|=4$ and
$\ellW(G)=2$ (Example~\ref{ex:s3}).

\subsection{The question}
The upper $\W$-series retains a direct connection with subnormality:
for every $H\sn G$, it gives a subnormal chain from $H^{\N}$ to $G$
of length at most $\ellW(G)$ (Theorem~\ref{thm:general-height}). Thus
$\ellW(G)$ uniformly bounds the subnormal defects of these residuals.
For soluble $G$ one also has $\ellW(G)\le h(G)$, where $h(G)$ is the
Fitting height. To determine how this normalizer condition controls
the global structure of $G$, we ask for the reverse comparison:
\begin{quote}
\emph{How much Fitting height can a single layer
$\W_{i+1}(G)/\W_i(G)$ of the upper $\W$-series carry?}
\end{quote}
Each layer is a $\W$-group, so the question is to bound the Fitting
height of soluble $\W$-groups. The answer is four. The smallest
example is the affine group
\[
\AGL(2,3)=\FF_3^2\rtimes\GL(2,3)
\]
of order $432$, whose Fitting layers are $\FF_3^2$, $Q_8$, $C_3$ and
$C_2$ (Proposition~\ref{prop:agl}). Height four always arises in this
way: from a primitive quotient whose point stabilizer has a
quaternion subgroup $Q_8$ in its Fitting subgroup and induces on it
the full group $\Out(Q_8)\cong S_3$.

\subsection{Main results}
\begin{theorem}\label{thm:main-height}
Every finite soluble group $G$ satisfies
\[
\ellW(G)\le h(G)\le4\ellW(G).
\]
If $|G|$ is odd or the Sylow $2$-subgroups of $G$ are abelian, then
$h(G)\le3\ellW(G)$. Conversely, for all integers $n\ge1$ and $h$ with
$n\le h\le4n$ there is a soluble $\{2,3,7,29\}$-group $G$ with
$\ellW(G)=n$ and $h(G)=h$. The maximum $4n$ is attained by
$\{2,3\}$-groups, and the maximum $3n$ is attained both by
$\{3,7,29\}$-groups and by soluble $A$-groups of odd order.
\end{theorem}
The upper bounds are Theorem~\ref{thm:soluble-height} and
Corollary~\ref{cor:soluble-length}; the constructions are
Theorems~\ref{thm:soluble-exact}, \ref{thm:odd-exact}
and~\ref{thm:all-pairs}. A group with $h(G)\ge2$ is not of
prime-power order, so two primes are needed for every pair with
$h\ge2$. Theorem~\ref{thm:main-height} should be compared with the
Bryce--Cossey bound $h(G)\le\wl(G)+1$: for the $\{2,3\}$-groups $G_n$
with $\ellW(G_n)=n$ and $h(G_n)=4n$ of
Theorem~\ref{thm:soluble-exact} it gives
\[
 \frac{\wl(G_n)}{\ellW(G_n)}\ge4-\frac1n.
\]
This is a lower bound for the ratio; we do not determine the exact
Wielandt lengths of these groups.

The upper bound rests on a description of the soluble primitive
$\W$-groups. In it, the module enters only through faithfulness and
irreducibility.
\begin{theorem}\label{thm:main-quaternion}
Let $G=V\rtimes M$, where $V\ne0$ is an elementary abelian $p$-group
and $M$ is a soluble group acting faithfully and irreducibly on $V$.
Then $\W(G)=G$ if and only if every abelian subgroup of $\Fit(M)$ is
cyclic, that is, the Sylow subgroups of $\Fit(M)$ are cyclic or
generalized quaternion. In this case $\W(M)=M$, $h(G)\le4$, and
$h(G)=4$ if and only if
\[
\Fit(M)=C\times Q_8,\qquad C\text{ cyclic of odd order},
\]
and conjugation maps $M/\Fit(M)$ onto $\Out(Q_8)\cong S_3$. Every
soluble $\W$-group of height four has a primitive quotient of this
form.
\end{theorem}
These assertions are Theorems~\ref{thm:affine-criterion}
and~\ref{thm:primitive-height}, and Theorem~\ref{thm:soluble-height}.
The implication from the structure of $\Fit(M)$ to $\W(M)=M$ is
Proposition~\ref{prop:rank-one}. For reducible modules we only use
the sufficient condition of Lemma~\ref{lem:affine-lifting}.

To show that the bounds are attained at every length, we stack
$\W$-groups by regular wreath products $B\wr E=B^E\rtimes E$. Ideally
the base $B^E$ is exactly the first term of the $\W$-series, so that
the length increases by one and the heights add. For soluble $B$, let
$\Sigma(B)$ be the set of primes $r$ such that $B$ has a primitive
quotient $U\rtimes P$, with $P\ne1$ an $r$-group acting faithfully and
irreducibly on $U$.
\begin{theorem}\label{thm:main-wreath}
Let $B,E$ be soluble with $B$ nonnilpotent. Then
$\W(B\wr E)\le B^E$, so that $\ellW(B\wr E)\ge1+\ellW(E)$. If
moreover $\W(B)=B$, then
\[
\W(B\wr E)=B^E\quad\Longleftrightarrow\quad
\Sigma(B)\cap\pi(\Fit(E))=\varnothing,
\]
and in this case $\ellW(B\wr E)=1+\ellW(E)$ and
$h(B\wr E)=h(B)+h(E)$. For arbitrary soluble $B\ne1$ and $E$, put
$b=h(B)$ and $D=B/\Fit_{b-1}(B)$. Then
\[
h(B\wr E)=
\begin{cases}
b+h(E),&D\text{ is not of prime-power order},\\
b+h(E/O_p(E)),&D\text{ is a }p\text{-group}.
\end{cases}
\]
\end{theorem}
These are Lemma~\ref{lem:base-upper} and
Theorems~\ref{thm:wreath-base} and~\ref{thm:wreath-height}. Since
$\Sigma(B)\subseteq\pi(B/B')$, the condition
$\gcd(|B/B'|,|\Fit(E)|)=1$ is sufficient. It is not necessary: for
the base $C_7\rtimes C_6$ every soluble top group is allowed
(Example~\ref{ex:c7c6}).

In Section~\ref{sec:extremal} we also note that in a soluble group with
$h(G)=4\ellW(G)$ every layer of the $\W$-series has height four, and
that equality can be read off from a single characteristic subgroup
(Proposition~\ref{prop:equality}).

\subsection{Strategy of the proofs}
\emph{Reduction to primitive groups.} Quotients of $\W$-groups are
$\W$-groups (Theorem~\ref{thm:functor}), and the Fitting height of a
soluble group is the maximum of the heights of its primitive
quotients. It therefore suffices to bound the height of a soluble
primitive $\W$-group $G=V\rtimes M$.

\emph{The affine criterion.} Let $1\ne x\in\Fit(M)$. The subgroup
$[V,x]\langle x\rangle$ is subnormal in $G$, and its nilpotent
residual is $[V,x]$. If $G$ is a $\W$-group, this residual is normal
in $G$, hence equals $V$ by irreducibility, and $x$ acts
fixed-point-freely. This forces every abelian subgroup of $\Fit(M)$
to be cyclic. Conversely, if every abelian subgroup of $\Fit(M)$ is
cyclic, then every element of prime order of $\Fit(M)$ generates a
normal subgroup of $M$ and hence acts fixed-point-freely, and the
residuals of the subnormal subgroups of $G$ are controlled by $M$.

\emph{Counting layers.} The quotient $M/\Fit(M)$ embeds in
$\Out(\Fit(M))$. For a nilpotent group all of whose abelian subgroups
are cyclic, this outer automorphism group has height at most two, with
equality only when a factor $Q_8$ occurs and $M$ induces all of
$\Out(Q_8)\cong S_3$. Counting the layers $V$, $\Fit(M)$ and
$M/\Fit(M)$ gives $h(G)\le1+1+2=4$.

\emph{Lengths and sharpness.} Subadditivity of Fitting height along the
upper $\W$-series gives $h(G)\le4\ellW(G)$. For sharpness we iterate
wreath products of $\AGL(2,3)$ with itself, using
Theorem~\ref{thm:main-wreath}: since $\Sigma(\AGL(2,3))=\{2\}$ and the
Fitting subgroups of the iterated groups are $3$-groups, each step adds
one layer to the series and four to the height.

\subsection{Related work and organization}
Several authors have studied the intersection of the normalizers of
the nilpotent residuals of \emph{all} subgroups of a group
\cite{GongGuo2013,GongGuo2016,ShenShiQian}. Gong and
Isaacs~\cite{GongIsaacs} proved that a subgroup normalizing the
nilpotent residuals of all non-subnormal subgroups also normalizes
those of the subnormal subgroups, and Ballester-Bolinches, Kamornikov
and Meng \cite{BKM} extended this to subgroup-closed saturated
formations containing the nilpotent groups, answering a question of
Gong and Isaacs. Hu, Huang and Skiba \cite{HuHuangSkiba} studied the
$\sigma$-nilpotent norm $N_\sigma(G)$, defined through all subgroups,
its iterated terminal term $N_\sigma^\infty(G)$, and the
$\sigma$-nilpotent length $l_\sigma(G)$. Their Theorem~1.4(i) says
that $G$ is $\sigma$-soluble with
$l_\sigma(G)\le r$, for $r>1$, if and only if
$G/N_\sigma^\infty(G)$ has these properties. For the partition of
the primes into singletons, it follows in particular that, for
soluble $G$, the norm of the nilpotent residuals of all subgroups
never increases the Fitting height beyond
$\max\{2,h(G/N_\sigma(G))\}$. This is in sharp contrast with
$\W(G)$, whose layers can carry four Fitting layers. The behaviour of
residuals of joins of subnormal subgroups, which underlies
Lemma~\ref{lem:Fstar}, has been studied for formations by
Ballester-Bolinches, Kamornikov and P\'erez-Calabuig \cite{BKPC}.

Section~\ref{sec:prelim} collects notation and standard facts, and
Section~\ref{sec:structure} the basic properties of $\W$ and of its
upper series. Section~\ref{sec:rank-one} studies soluble groups whose
Fitting subgroup has only cyclic abelian subgroups, and
Section~\ref{sec:affine} applies this to primitive groups and proves
the height bounds. Section~\ref{sec:construction} gives the examples
of length one, Section~\ref{sec:soluble-wreath} treats wreath
products, and Section~\ref{sec:extremal} constructs the extremal
groups and all pairs $(\ellW,h)$. The appendix treats groups with
trivial soluble radical and nonsoluble length; it depends on the
classification of finite simple groups and is not used in the body of
the paper. A companion paper \cite{DefectCompanion} studies residual
subnormal defect two in affine groups; none of the proofs here
depends on it.

A reader interested only in the upper bound $h(G)\le4\ellW(G)$ can
read Sections~\ref{sec:prelim}--\ref{sec:affine}. Sections~\ref{sec:construction}--\ref{sec:extremal}
show that the bound is sharp, and Table~\ref{tab:examples} lists the
examples.

\section{Preliminaries}
\label{sec:prelim}

\subsection{Notation}
Write $\Fit(G)$ for the Fitting subgroup and $\Fit_i(G)$ for the upper
Fitting series, with $\Fit_0(G)=1$; for soluble $G$ its length is the
Fitting height $h(G)$, with $h(1)=0$. An \emph{$A$-group} is a group
with abelian Sylow subgroups, $\pi(G)$ is the set of prime divisors of
$|G|$, and $O_p(G)$ is the largest normal $p$-subgroup. Regular wreath
products are written $B\wr E=B^E\rtimes E$, where $E$ permutes the
coordinate set $E$ regularly. Modules are right modules, with
$[v,g]=v^g-v$ when the underlying group is written additively. We
write $\Layer(G)$ for the layer and $\Fit^*(G)=\Fit(G)\Layer(G)$ for
the generalized Fitting subgroup; the length of the upper generalized
Fitting series is $h^*(G)$, and $h^*(G)=h(G)$ for soluble $G$.

For $L\sn G$, let $\Def_G(L)$ be the usual subnormal defect of $L$ in
$G$, that is, the minimum length of a subnormal chain from $L$ to $G$.
Put $\bN(G)=\max_{H\sn G}\Def_G(H^{\N})$.

\subsection{Basic facts}
Write $[A,{}_dB]$ for $d$ successive commutators with $B$, where
$[A,{}_0B]=A$. Since the nilpotent residual is the stable term of the
lower central series,
\begin{equation}\label{eq:stable}
 [G^{\N},G]=G^{\N}.
\end{equation}
\begin{lemma}\label{lem:residual-basics}
The following statements hold.
\begin{enumerate}
\renewcommand{\labelenumi}{\textup{(\roman{enumi})}}
\item If $H\le G$, then $H^{\N}\le G^{\N}$.
\item If $f:G\twoheadrightarrow Q$, then $f(G^{\N})=Q^{\N}$.
In particular, $(HN/N)^{\N}=H^{\N}N/N$ for $N\trianglelefteq G$.
\item If $H\sn Y$ has defect at most $d$, then $[Y,{}_dH]\le H$.
\end{enumerate}
\end{lemma}
\begin{proof}
For (i), $H/(H\cap G^{\N})$ embeds in the nilpotent group
$G/G^{\N}$. For (ii), both $Q/f(G^{\N})$ and
$G/f^{-1}(Q^{\N})$ are nilpotent, giving the two inclusions.
For (iii), choose a chain
$H=H_0\trianglelefteq\cdots\trianglelefteq H_d=Y$.
The inclusions $[H_i,H]\le H_{i-1}$ for $1\le i\le d$
give the assertion by induction.
\end{proof}

\begin{lemma}\label{lem:Fstar}
For every $H\sn G$, both $\Fit(G)$ and $\Layer(G)$ normalize
$H^{\N}$, and $(\Fit(G)H)^{\N}=H^{\N}$. In particular,
\begin{equation}\label{eq:Fstar}
 \Fit^*(G)\le\W(G).
\end{equation}
\end{lemma}
\begin{proof}
We use the nilpotent case of the residual-join theorem in
\cite[Theorem~2.5(i), p.~420]{BE}: if $A,B$ are subnormal in a
finite group and $A$ is nilpotent, then
$\langle A,B\rangle^{\N}=B^{\N}$. Indeed, apply that theorem to
the Fitting formation $\N$, using that a nilpotent group is
quasinilpotent and $A^{\N}=1$. No solubility hypothesis is needed
on the ambient group. Taking $A=\Fit(G)$ and $B=H$ gives
$(\Fit(G)H)^{\N}=H^{\N}$. The latter subgroup is characteristic in
$\Fit(G)H$ and hence normalized by $\Fit(G)$.
For the layer, use the standard component alternative
\cite[Chapter~9]{Isaacs}: if $L$ is a component and $H\sn G$, then
either $L\le H$ or $[L,H]=1$.
In the first case $L$ normalizes $H^{\N}$ because $H^{\N}$ is
characteristic in $H$; in the second case it centralizes it.
Thus every component normalizes $H^{\N}$, as does their product
$\Layer(G)$. Intersecting over $H\sn G$ proves \eqref{eq:Fstar}.
\end{proof}

We use the self-centralizing property $C_G(\Fit^*(G))\le\Fit^*(G)$
\cite[Chapter~9]{Isaacs} and the following standard results of finite
soluble group theory \cite{DH,Isaacs}.
\begin{enumerate}
\renewcommand{\labelenumi}{\textup{(S\arabic{enumi})}}
\item Fitting height does not increase on taking subgroups or quotients.
For $N\trianglelefteq G$,
\[
 h(G)\le h(N)+h(G/N).
\]
The height of a direct product is the maximum of the heights of its factors.
\item Fitting height is unchanged on factoring out the Frattini subgroup:
\[
 h(G/\Phi(G))=h(G).
\]
\item If $G$ is soluble, then $C_G(\Fit(G))\le\Fit(G)$.
\item (Burnside) A finite $q$-group containing no subgroup isomorphic to
$C_q\times C_q$ is cyclic when $q$ is odd, and cyclic or generalized
quaternion when $q=2$. Conversely, cyclic and generalized quaternion
groups have a unique subgroup of prime order.
\end{enumerate}
In particular, all abelian subgroups of a nilpotent group $D$ are
cyclic if and only if $D$ has no subgroup $C_q\times C_q$ for any
prime $q$, if and only if the Sylow subgroups of $D$ are cyclic or
generalized quaternion.

\begin{lemma}[Height is detected by primitive quotients]
\label{lem:primitive-quotients}
Let $G\ne1$ be soluble. For a maximal subgroup $L$ of $G$, the
quotient $Y=G/\Core_G(L)$ has the form $Y=V\rtimes J$, where $V$ is
an elementary abelian minimal normal subgroup of $Y$ and
$J=L/\Core_G(L)$ acts faithfully and irreducibly on $V$. Moreover,
\begin{equation}\label{eq:primitive-max}
h(G)=\max_{L\text{ maximal in }G}h\bigl(G/\Core_G(L)\bigr).
\end{equation}
\end{lemma}
\begin{proof}
Put $K=\Core_G(L)$ and $J=L/K$, a core-free maximal subgroup of
$Y=G/K$. Choose a minimal normal subgroup $V$ of $Y$; it is
elementary abelian because $Y$ is soluble. Core-freeness implies
$V\not\le J$, so $Y=VJ$ by maximality. The subgroup $V\cap J$ is
normalized by $J$ and centralized by $V$, and hence is normal in $Y$;
thus it is trivial. Also $C_J(V)$ is normalized by $J$ and
centralized by $V$, and is therefore normal in $Y$ and contained in
$J$; hence $C_J(V)=1$. Finally, every $J$-invariant subspace of $V$
is normal in $Y$, so the action is irreducible.

The intersection of the cores of all maximal subgroups is $\Phi(G)$:
it is contained in the intersection of the maximal subgroups, while
the normal subgroup $\Phi(G)$ is contained in every core. Therefore
$G/\Phi(G)$ embeds as a subdirect product of the quotients
$G/\Core_G(L)$, and \eqref{eq:primitive-max} follows from (S1) and (S2).
\end{proof}

The proofs in Sections~\ref{sec:structure}--\ref{sec:extremal} do not
use the classification of finite simple groups.

\section{The subgroup $\mathrm W(G)$ and its upper series}
\label{sec:structure}
This section develops the basic theory of $\W(G)$. Theorem~\ref{thm:structure}
shows that $\W(G)$ is a characteristic $\W$-group containing the
generalized Fitting subgroup. Theorem~\ref{thm:functor} shows that the
construction is compatible with subnormal subgroups and quotients.
Theorem~\ref{thm:general-height} shows that the upper $\W$-series gives
subnormal chains for the residuals and is bounded by the Fitting
height. These facts are used throughout the rest of the paper.

\begin{theorem}\label{thm:structure}
Let $X=\W(G)$. Then:
\begin{enumerate}
\renewcommand{\labelenumi}{\textup{(\roman{enumi})}}
\item $X$ is characteristic in $G$, and
$\Fit^*(G)\le X$ and $\omega(G)\le X$.
\item $C_G(X)=Z(X)$.
\item $\Fit(X)=\Fit(G)$, $\Layer(X)=\Layer(G)$, and
$\Fit^*(X)=\Fit^*(G)$.
\item $\W(X)=X$.
\end{enumerate}
\end{theorem}
\begin{proof}
Automorphisms permute subnormal subgroups and map $H^{\N}$ to
$(H^\alpha)^{\N}$. They therefore permute the normalizers defining
$X$, proving characteristicity. The containment of $\Fit^*(G)$
is Lemma~\ref{lem:Fstar}.
If $x$ normalizes $H$, it normalizes the characteristic subgroup
$H^{\N}$, so $\omega(G)\le X$. Since $\Fit^*(G)\le X$,
\[
 C_G(X)\le C_G(\Fit^*(G))\le\Fit^*(G)\le X,
\]
whence $C_G(X)=Z(X)$.

The subgroup $\Fit(X)$ is characteristic in $X$ and hence normal in
$G$, so $\Fit(X)\le\Fit(G)$. Conversely, $\Fit(G)\le X$ is
nilpotent and normal in $X$. This proves equality.
A component of $X$ is quasisimple and subnormal in $G$, and is
therefore a component of $G$. Conversely, each component of $G$ lies
in $\Layer(G)\le X$; intersecting a subnormal chain with $X$ shows
that it is subnormal in $X$. Thus $\Layer(X)=\Layer(G)$, and the
equality of generalized Fitting subgroups follows.

If $H\sn X$, then $H\sn G$ because $X\trianglelefteq G$.
By definition $X$ normalizes $H^{\N}$. Hence $\W(X)=X$.
\end{proof}

Each $\W_i(G)$ is characteristic in $G$. If $\W_i(G)<G$, the
nontrivial group $G/\W_i(G)$ has a nontrivial generalized Fitting
subgroup, so $\W_i(G)<\W_{i+1}(G)$ by Theorem~\ref{thm:structure}(i).
Thus the upper $\W$-series strictly increases until it reaches $G$,
and $\ellW(G)$ is at most the number of prime factors of $|G|$,
counted with multiplicity.

\begin{theorem}\label{thm:functor}
If $K\sn G$ and $N\trianglelefteq G$, then
\[
 K\cap\W(G)\le\W(K),\qquad
 \W(G)N/N\le\W(G/N).
\]
Equivalently, every epimorphism $f:G\twoheadrightarrow Q$ satisfies
$f(\W(G))\le\W(Q)$. In particular, subnormal subgroups and quotients
of $\W$-groups are $\W$-groups.
\end{theorem}
\begin{proof}
Every subnormal subgroup of $K$ is subnormal in $G$, giving the first
containment. For the second, lift $L/N\sn G/N$ to $L\sn G$.
By Lemma~\ref{lem:residual-basics}, $(L/N)^{\N}=L^{\N}N/N$.
Every element of $\W(G)$ normalizes $L^{\N}$ and hence normalizes
this image. Taking intersections proves the assertion. The epimorphism
form follows by factoring through the kernel.
\end{proof}

\begin{lemma}
\label{lem:residual-product}
Suppose $H\sn A\times B$, and put
$A_0=\pi_A(H)$ and $B_0=\pi_B(H)$. Then
\[
H^{\N}=A_0^{\N}\times B_0^{\N}.
\]
\end{lemma}
\begin{proof}
Set $Y=A_0\times B_0$. Since $H\le Y$, intersecting a subnormal
chain with $Y$ shows that $H\sn Y$. Choose
$d\ge\max\{1,\Def_Y(H)\}$.
By Lemma~\ref{lem:residual-basics}(iii), $[Y,{}_dH]\le H$.
Because $H$ projects onto $A_0$, coordinatewise commutator calculation
gives
\[
 [A_0\times1,{}_dH]=\gamma_{d+1}(A_0)\times1\le H.
\]
In particular $R=A_0^{\N}\times1\le H$. By
\eqref{eq:stable} and the same projection argument,
\[
 [R,H]=[A_0^{\N},A_0]\times1=R.
\]
Starting with $R\le H$, induction gives $R\le\gamma_j(H)$ for
every $j$, and therefore $R\le H^{\N}$.
Similarly $1\times B_0^{\N}\le H^{\N}$.

Conversely, $H\le Y$ and Lemma~\ref{lem:residual-basics}(i) give
$H^{\N}\le Y^{\N}$.
The lower central series of a direct product is computed componentwise,
so $Y^{\N}=A_0^{\N}\times B_0^{\N}$. Both containments now follow.
\end{proof}

\begin{theorem}\label{thm:W-product}
For any finite groups $A$ and $B$,
\[
\W(A\times B)=\W(A)\times\W(B).
\]
\end{theorem}
\begin{proof}
If $H\sn A\times B$, its projections $A_0$ and $B_0$ are
subnormal in their respective factors. Thus $\W(A)$ normalizes
$A_0^{\N}$ and $\W(B)$ normalizes $B_0^{\N}$.
Lemma~\ref{lem:residual-product} shows that $\W(A)\times\W(B)$ normalizes
$H^{\N}$. Taking intersections gives one containment.
For the reverse containment, apply the definition to every subgroup
$K\times1$, $K\sn A$, and every subgroup $1\times L$, $L\sn B$.
\end{proof}

\begin{lemma}\label{lem:series-shift}
For $i,j\ge0$,
\[
 \W_j(G/\W_i(G))=\W_{i+j}(G)/\W_i(G).
\]
Consequently, $\ellW(G/\W_i(G))=\ellW(G)-i$ for
$0\le i\le\ellW(G)$. Moreover, every epimorphism
$f:G\twoheadrightarrow Q$ satisfies $f(\W_i(G))\le\W_i(Q)$ for all
$i$, and hence $\ellW(Q)\le\ellW(G)$.
\end{lemma}
\begin{proof}
The subgroup identity follows by induction on $j$ from the defining
recursion. The length identity follows from strict increase before
termination. The last assertion follows by induction on $i$, applying
Theorem~\ref{thm:functor} to the induced epimorphisms
$G/\W_i(G)\twoheadrightarrow Q/\W_i(Q)$.
\end{proof}

\begin{theorem}
\label{thm:general-height}
Let $n=\ellW(G)$. For every $H\sn G$, with $R=H^{\N}$,
\[
 R\trianglelefteq R\W_1(G)\trianglelefteq\cdots
 \trianglelefteq R\W_n(G)=G.
\]
Moreover, $\Fit_i^*(G)\le\W_i(G)$ and $\omega_i(G)\le\W_i(G)$ for every
$i\ge0$. Consequently,
\[
\bN(G)\le\ellW(G)\le\min\{h^*(G),\wl(G)\},
\]
and $\bN(G)\le\ellW(G)\le h(G)$ if $G$ is soluble.
\end{theorem}
\begin{proof}
The subgroup $H\W_i(G)/\W_i(G)$ is subnormal in $G/\W_i(G)$, and its
nilpotent residual is $R\W_i(G)/\W_i(G)$ by
Lemma~\ref{lem:residual-basics}. By definition,
$\W_{i+1}(G)/\W_i(G)$ normalizes this residual, so $\W_{i+1}(G)$
normalizes $R\W_i(G)$. This proves the displayed chain, and hence
$\Def_G(R)\le n$.

Suppose $\Fit_i^*(G)\le\W_i(G)$. By Theorem~\ref{thm:functor}, the
natural epimorphism $G/\Fit_i^*(G)\twoheadrightarrow G/\W_i(G)$ maps
$\W(G/\Fit_i^*(G))$ into $\W(G/\W_i(G))$, and the former contains
$\Fit^*(G/\Fit_i^*(G))$ by Lemma~\ref{lem:Fstar}. Taking inverse
images gives $\Fit_{i+1}^*(G)\le\W_{i+1}(G)$. The same argument,
with $\omega(X)\le\W(X)$ from Theorem~\ref{thm:structure}(i), gives
$\omega_i(G)\le\W_i(G)$. These inclusions give the upper bounds.
\end{proof}

In words, the upper $\W$-series plays for the residuals $H^{\N}$ the
role that the Wielandt series plays for the subnormal subgroups
themselves, and its length is at most the Fitting height. The rest of
the paper studies the reverse comparison.

\section{Soluble groups whose Fitting subgroup has cyclic abelian subgroups}
\label{sec:rank-one}
In this section $M$ is soluble and every abelian subgroup of
$D=\Fit(M)$ is cyclic. By (S4) we may write $D=C\times P$, where $C$
is cyclic of odd order and $P$ is the Sylow $2$-subgroup of $D$, which
is cyclic or generalized quaternion. We show that $M$ is a
$\W$-group and that $h(M)\le3$. These groups are the point stabilizers
of the primitive $\W$-groups (Theorem~\ref{thm:affine-criterion}), and
the bound $h(M)\le3$ is the source of the constant four. The key point
is that $M/\Fit(M)$ embeds in $\Out(\Fit(M))$, which has height at most
two, with equality only when a factor $Q_8$ occurs
(Lemmas~\ref{lem:out-quaternion} and~\ref{lem:embedding}).

\begin{lemma}\label{lem:out-quaternion}
Let $D=C\times P$ be as above. Then
\[
 \Out(D)\cong\Aut(C)\times\Out(P),
\]
the group $\Aut(C)$ is abelian, and $\Aut(P)$ is a $2$-group unless
$P\cong Q_8$. Moreover $\Aut(Q_8)\cong S_4$ and $\Out(Q_8)\cong S_3$.
In particular $h(\Out(D))\le2$, and $\Out(D)$ is nilpotent unless
$P\cong Q_8$.
\end{lemma}
\begin{proof}
The factors $C$ and $P$ are characteristic of coprime orders, which
gives the decomposition, and $\Aut(C)$ is abelian because $C$ is
cyclic. If $P$ is cyclic, $\Aut(P)$ is a $2$-group. Let
$P=Q_{2^m}$ be generalized quaternion of order $2^m$, $m\ge3$.

If $m=3$, the action on the three cyclic subgroups of order four gives
a homomorphism $\Aut(Q_8)\to S_3$. It is onto, and its kernel consists
of the four automorphisms $i\mapsto\pm i$, $j\mapsto\pm j$, which are
exactly the inner automorphisms. Hence $|\Aut(Q_8)|=24$,
$\Out(Q_8)\cong S_3$, and the action on the six elements of order
four identifies $\Aut(Q_8)$ with $S_4$.

If $m\ge4$, write
\[
 P=\langle a,b\mid a^{2^{m-1}}=1,\ b^2=a^{2^{m-2}},\ a^b=a^{-1}\rangle.
\]
The subgroup $\langle a\rangle$ is the unique cyclic subgroup of index
two. An automorphism is determined by
$a\mapsto a^u$, $b\mapsto a^v b$, with $u$ odd; all these choices
give automorphisms. Their number is $2^{m-2}2^{m-1}$, so $\Aut(P)$
is a $2$-group.

Consequently $\Out(D)$ is the direct product of an abelian group and
a $2$-group, and is nilpotent, unless $P\cong Q_8$; in that case it is
the direct product of an abelian group and $S_3$. In both cases
$h(\Out(D))\le2$.
\end{proof}

\begin{lemma}\label{lem:embedding}
Let $M$ and $D=C\times P$ be as above, and put $E=M/D$. Conjugation
induces an embedding
\[
 E\hookrightarrow\Aut(C)\times\Out(P).
\]
If $P\not\cong Q_8$, or if $P\cong Q_8$ and the induced homomorphism
$\theta:E\to\Out(P)\cong S_3$ is not surjective, then $E$ is
nilpotent. If $P\cong Q_8$ and $\theta$ is surjective, then
\[
 h(E)=2,\quad E'\cong C_3,
 \quad \Fit(E)=\theta^{-1}(A_3),\quad |E:\Fit(E)|=2,
\]
$\Fit(E)$ is abelian, and every element of $E\setminus\Fit(E)$ inverts
$E'$.
\end{lemma}
\begin{proof}
By (S3), $C_M(D)\le D$. An element of $M$ induces an inner
automorphism of $D$ exactly when it lies in $DC_M(D)=D$, so
conjugation embeds $E$ in $\Out(D)$, and Lemma~\ref{lem:out-quaternion}
gives the displayed embedding, whose second coordinate is $\theta$.
If $P\not\cong Q_8$, the target is nilpotent. If $P\cong Q_8$ and
$\theta(E)$ is a proper, hence cyclic, subgroup of $S_3$, then $E$
embeds in the abelian group $\Aut(C)\times\theta(E)$.

Assume that $P\cong Q_8$ and $\theta$ is onto, and identify $E$ with
its image. The first coordinate group is abelian, so
$E'\le1\times A_3$; surjectivity gives $\theta(E')=S_3'=A_3$, and
therefore $E'=1\times A_3\cong C_3$. The subgroup
$E_0=\theta^{-1}(A_3)$ has index two and lies in the abelian group
$\Aut(C)\times A_3$, so $E_0\le\Fit(E)$. An element $x\in
E\setminus E_0$ has image a transposition, and so inverts $E'$. On the
other hand, $E'$ is a normal subgroup of order three of the nilpotent
group $\Fit(E)$, hence central in $\Fit(E)$. Therefore
$\Fit(E)=E_0$. Finally $E$ is not nilpotent, since it maps onto $S_3$,
and $E/E_0$ is abelian, so $h(E)=2$.
\end{proof}

\begin{proposition}\label{prop:rank-one}
Let $M$ be a soluble group such that every abelian subgroup of
$\Fit(M)$ is cyclic. Then $\W(M)=M$ and $h(M)\le3$. Moreover
$h(M)=3$ if and only if $\Fit(M)=C\times Q_8$ with $C$ cyclic of odd
order and conjugation maps $M/\Fit(M)$ onto $\Out(Q_8)\cong S_3$.
\end{proposition}
\begin{proof}
Keep the notation $D=C\times P$ and $E=M/D$. The assertions about
height follow from Lemma~\ref{lem:embedding}, since $h(M)=1+h(E)$
when $M\ne1$.

Let $H\sn M$. The subgroup $DH$ is subnormal in $M$, being the
inverse image of the subnormal subgroup $DH/D$ of $E$, and
$(DH)^{\N}=H^{\N}$ by Lemma~\ref{lem:Fstar}. We may therefore assume
that $D\le H$, so that $H$ is the inverse image of the subnormal
subgroup $\overline H=H/D$ of $E$. Put $R=H^{\N}$.

Suppose first that $\overline H$ is nilpotent. Then $R\le D$, and
since $C$ and $P$ have coprime orders,
\[
 R=A\times B,\qquad A=R\cap C,\quad B=R\cap P.
\]
Every subgroup of the cyclic group $C$ is characteristic in $C$, so
$A\trianglelefteq M$. The subgroup $B$ is normal in $H$, and
\eqref{eq:stable} gives $R=[R,H]=[A,H]\times[B,H]$, whence
$B=[B,H]$. We claim that $B=1$ or $B=P\cong Q_8$; in both cases
$B\trianglelefteq M$, and therefore $R\trianglelefteq M$. The group
$B$ is cyclic or generalized quaternion, as a subgroup of $P$. Let
$X$ be the image of $H$ in $\Aut(B)$. If $X$ is a $2$-group, then
$[B,H]$ is contained in $[B,B\rtimes X]$, which is a proper subgroup
of $B$ unless $B=1$, because $B\rtimes X$ is a $2$-group; hence $B=1$.
By Lemma~\ref{lem:out-quaternion}, $X$ is a $2$-group unless
$B\cong Q_8$. If $B\cong Q_8$ and $B<P$, then $P$ is generalized
quaternion of order at least $16$, so the image of $M$ in $\Aut(P)$
is a $2$-group by Lemma~\ref{lem:out-quaternion}; the image of $H$ in
$\Aut(B)$ is a quotient of a subgroup of it, hence again a $2$-group,
and $B=1$, a contradiction. This proves the claim.

Suppose now that $\overline H$ is not nilpotent. Then $E$ is not
nilpotent, so by Lemma~\ref{lem:embedding} we have $P\cong Q_8$,
$\theta$ is onto, and $E_0=\Fit(E)$ is abelian. Thus
$\overline H\not\le E_0$, and $\overline H$ contains an element $t$
inverting $E'\cong C_3$. Then $[E',{}_dt]=E'$ for every $d$, and
Lemma~\ref{lem:residual-basics}(iii) gives $E'\le\overline H$. Since
$E/E'$ is abelian, $\overline H\trianglelefteq E$, so
$H\trianglelefteq M$ and its characteristic subgroup $R$ is normal in
$M$.

In either case $R\trianglelefteq M$, so $\W(M)=M$.
\end{proof}

\section{Primitive affine $\mathrm W$-groups and Fitting height}
\label{sec:affine}
We now prove the height bounds. Section~\ref{subsec:affine-criterion}
characterizes the soluble primitive $\W$-groups $V\rtimes M$: whether
$V\rtimes M$ is a $\W$-group depends only on $\Fit(M)$
(Theorem~\ref{thm:affine-criterion}). Section~\ref{subsec:height-bounds}
combines this with Section~\ref{sec:rank-one} to bound their height by
four and to identify the groups of height four. It then passes from
primitive groups to arbitrary soluble groups, and from a single layer
to the whole upper $\W$-series.

\subsection{The affine criterion}
\label{subsec:affine-criterion}
Lemmas~\ref{lem:subnormal-fitting} and~\ref{lem:affine-fitting} are
standard facts about subnormal subgroups and affine groups.
Lemma~\ref{lem:affine-lifting} is our main tool for producing
$\W$-groups: if a normal subgroup $Q$ of $M$ acts fixed-point-freely
and meets every nontrivial subnormal subgroup, then $V\rtimes M$
inherits the $\W$-property from $M$.
\begin{lemma}\label{lem:subnormal-fitting}
Let $A\sn X$. Then $\Fit(A)\le\Fit(X)$, and if $A$ is an $r$-group for
a prime $r$, then $A\le O_r(X)$. Consequently, if $X$ is soluble and
$1\ne A\sn X$, then $A\cap\Fit(X)\ne1$.
\end{lemma}
\begin{proof}
A nilpotent subnormal subgroup $T$ of $X$ lies in $\Fit(X)$, and a
subnormal $r$-subgroup lies in $O_r(X)$. To see this, induct on the
length of a subnormal chain. The normal case follows from the
definitions. Otherwise take the penultimate term $J\trianglelefteq X$
of the chain. Induction gives $T\le\Fit(J)$, respectively
$T\le O_r(J)$. These subgroups are characteristic in $J$, hence
normal in $X$, and so are contained in $\Fit(X)$, respectively
$O_r(X)$.

Now $\Fit(A)$ is characteristic in $A$, and is therefore nilpotent
and subnormal in $X$. If $A\ne1$ is soluble, then $\Fit(A)\ne1$, and
$1\ne\Fit(A)\le A\cap\Fit(X)$.
\end{proof}

\begin{lemma}\label{lem:affine-fitting}
Let $V\ne0$ be a faithful $\FF_pM$-module and put $G=V\rtimes M$. If
$O_p(M)=1$, then
\[
\Fit(G)=V,\qquad h(G)=1+h(M)
\]
when $M$ is soluble. The condition $O_p(M)=1$ holds in particular if
$V$ is irreducible.
\end{lemma}
\begin{proof}
Faithfulness gives $C_G(V)=V$. The Fitting subgroup contains
$V$, and its Hall $p'$-subgroup centralizes $V$, so that Hall subgroup
is trivial. The remaining $p$-subgroup has image contained in
$O_p(M)=1$, proving the Fitting identity. The height formula follows
by taking the quotient by $V$.

A finite $p$-group acting on a nonzero vector space over $\FF_p$ has
a nonzero common fixed vector: its non-singleton orbits have
cardinality divisible by $p$, so the number of common fixed vectors
is divisible by $p$ and includes zero. If $V$ is irreducible, the
nonzero fixed space of $O_p(M)$ is $M$-invariant, hence equals $V$,
and faithfulness gives $O_p(M)=1$.
\end{proof}

The next lemma does not require $V$ to be irreducible, and $p$ may
divide $|M|$.

\begin{lemma}[Affine lifting]\label{lem:affine-lifting}
Let $M$ be a group with $\W(M)=M$, and let $Q\trianglelefteq M$ meet
every nontrivial subnormal subgroup of $M$ nontrivially. Let $V\ne0$
be an $\FF_pM$-module such that
\[
 C_V(q)=0\qquad(1\ne q\in Q),
\]
and let $\pi:G=V\rtimes M\twoheadrightarrow M$ be the natural
projection. Then the following hold.
\begin{enumerate}
\renewcommand{\labelenumi}{\textup{(\roman{enumi})}}
\item $M$ acts faithfully on $V$, and $O_p(M)=1$.
\item Every $H\sn G$ satisfies exactly one of the following
alternatives:
\begin{enumerate}
\renewcommand{\labelenumii}{\textup{(\alph{enumii})}}
\item $H\le V$ and $H^{\N}=1$;
\item $H\not\le V$, $V\le H$, and, with $A=H\cap M$, one has
\[
 A\ne1,\qquad A\sn M,\qquad H=VA,\qquad H^{\N}=VA^{\N}.
\]
\end{enumerate}
\item $\W(G)=G$. If $M$ is soluble, then $\Fit(G)=V$ and
$h(G)=1+h(M)$.
\end{enumerate}
\end{lemma}
\begin{proof}
(i) The kernel $K$ of the action is normal in $M$, and $K\cap Q=1$
because every nonidentity element of $Q$ moves some vector. The
hypothesis on $Q$ forces $K=1$. Similarly, if $O_p(M)\ne1$, it
contains a nonidentity element $q\in Q$ of $p$-power order. As
recalled in the proof of Lemma~\ref{lem:affine-fitting}, $q$ fixes a
nonzero vector of $V$, contrary to $C_V(q)=0$. Hence $O_p(M)=1$.

(ii) Let $H\sn G$ and put $A=\pi(H)$, so that $A\sn M$. If $A=1$,
then $H\le V$ is abelian and $H^{\N}=1$, which is alternative (a).

Suppose $A\ne1$, choose $1\ne q\in A\cap Q$, and choose $y\in H$ with
$\pi(y)=q$. Since $V$ is abelian, conjugation by $y$ induces on $V$
the same linear map as $q$. The hypothesis $C_V(q)=0$ says that $q-1$
is injective, hence invertible, on $V$. Consequently, for every
$d\ge1$,
\begin{equation}\label{eq:invertible}
 [V,{}_dy]=(q-1)^dV=V.
\end{equation}
Choose $d\ge1$ at least the subnormal defect of $H$ in $G$.
Lemma~\ref{lem:residual-basics}(iii) and~\eqref{eq:invertible} give
\[
 V=[V,{}_dy]\le[G,{}_dH]\le H.
\]
Therefore $H$ is the full inverse image of $A$, so $H=VA$ and
$A=H\cap M$. Equation~\eqref{eq:invertible} with $d=1$ gives
$[V,H]=V$; since $V\le H$, induction yields $V\le\gamma_i(H)$ for
every $i$, and hence $V\le H^{\N}$. By Lemma~\ref{lem:residual-basics}(ii),
$\pi(H^{\N})=A^{\N}$. Because $H^{\N}$ contains the kernel $V$, it is
the full inverse image of $A^{\N}$, that is, $H^{\N}=VA^{\N}$. This is
alternative (b).

(iii) In alternative (b), $A\sn M$ and $\W(M)=M$ give
$A^{\N}\trianglelefteq M$, so its full inverse image $VA^{\N}$ is
normal in $G$. Together with alternative (a), this shows that
$H^{\N}\trianglelefteq G$ for every $H\sn G$; thus $\W(G)=G$. The last
assertion follows from (i) and Lemma~\ref{lem:affine-fitting}.
\end{proof}

\begin{lemma}\label{lem:fpf-abelian}
Let $A$ be an abelian group of order prime to $p$ acting on a
nonzero $\FF_p$-vector space $V$ with $C_V(x)=0$ for every
$1\ne x\in A$. Then $A$ is cyclic.
\end{lemma}
\begin{proof}
Since $A$ is abelian of order prime to $p$, its action is
simultaneously diagonalizable over an algebraic closure of $\FF_p$.
Choose a common eigenvector and let
$\chi:A\to\overline{\FF}_p^{\,\times}$ be its eigencharacter. If
$\ker\chi$ contained an element $x\ne1$, then $x$ would have
eigenvalue one on the scalar extension of $V$, contradicting the
invertibility of $x-1$ on $V$, which is preserved by extension of
scalars. Hence $\chi$ is injective, and $A$ is cyclic, as a finite
subgroup of the multiplicative group of a field.
\end{proof}

\begin{theorem}\label{thm:affine-criterion}
Let $V\ne0$ be a finite-dimensional $\FF_p$-vector space, and let a
finite soluble group $M$ act faithfully and irreducibly on $V$. Put
\[
 G=V\rtimes M,\qquad D=\Fit(M).
\]
The following conditions are equivalent:
\begin{enumerate}
\renewcommand{\labelenumi}{\textup{(\roman{enumi})}}
\item $\W(G)=G$;
\item $C_V(x)=0$ for every $1\ne x\in D$;
\item $C_V(x)=0$ for every element $x\in D$ of prime order;
\item every abelian subgroup of $D$ is cyclic.
\end{enumerate}
Whenever these conditions hold, $\W(M)=M$ and the subnormal subgroups
of $G$ and their residuals are described by
Lemma~\ref{lem:affine-lifting}(ii). In all cases $O_p(M)=1$,
$p\nmid|D|$, $\Fit(G)=V$, and $h(G)=1+h(M)$.
\end{theorem}
\begin{proof}
By Lemma~\ref{lem:affine-fitting}, $O_p(M)=1$, $\Fit(G)=V$ and
$h(G)=1+h(M)$. The Sylow $p$-subgroup of $D$ is characteristic in $D$,
hence normal in $M$ and contained in $O_p(M)=1$; thus $p\nmid|D|$.

(i)$\Rightarrow$(iii). We test the $\W$-group property on the
subnormal subgroups $[V,x]\langle x\rangle$, $x\in D$. Suppose that
$\W(G)=G$. Let $x\in D$ have prime order, and set
\[
 A=\langle x\rangle,\qquad U=[V,A],\qquad T=UA.
\]
Every subgroup of a finite nilpotent group is subnormal, since a
proper subgroup is properly contained in its normalizer. Hence
$A\sn D\trianglelefteq M$, and lifting a chain gives $VA\sn G$.
The subgroup $U$ is $A$-invariant. Conjugating an element of $A$ by
an element of $V$ changes it by an element of $[V,A]=U$, while $V$
centralizes $U$. Thus $T\trianglelefteq VA$ and $T\sn G$.

Since $p\nmid|A|$, coprime action on $V$ gives
\[
 V=C_V(A)\oplus U,\qquad [U,A]=U,
\]
for instance by means of the averaging projection
$|A|^{-1}\sum_{a\in A}a$ onto $C_V(A)$. Now $T/U\cong A$ is
nilpotent, so $T^{\N}\le U$. Conversely, $[U,T]=[U,A]=U$ implies by
induction that $U\le\gamma_i(T)$ for every $i$. Hence $T^{\N}=U$.
The $\W$-group hypothesis gives $U\trianglelefteq G$, so $U$ is an
$M$-invariant subspace of $V$, and $U=0$ or $U=V$. The first
possibility would make $x\ne1$ act trivially on $V$, contradicting
faithfulness. Thus $U=V$, and $C_V(x)=0$.

(iii)$\Rightarrow$(ii). If $1\ne x\in D$ has order $n$, choose a
prime $r\mid n$. The element $y=x^{n/r}$ has prime order, and
$C_V(x)\le C_V(y)=0$.

(ii)$\Rightarrow$(iv). An abelian subgroup of $D$ has order prime to
$p$, and Lemma~\ref{lem:fpf-abelian} shows that it is cyclic.

(iv)$\Rightarrow$(iii). Let $x\in D$ have prime order $r$, and let
$R$ be the Sylow $r$-subgroup of $D$. By (S4), $\langle x\rangle$ is
the unique subgroup of order $r$ of $R$, so it is characteristic in
$R$, hence in $D$, and normal in $M$. Consequently $C_V(x)$ is an
$M$-invariant subspace of $V$, and $C_V(x)\ne V$ by faithfulness.
Irreducibility gives $C_V(x)=0$.

Finally, assume (ii)--(iv). Proposition~\ref{prop:rank-one} gives
$\W(M)=M$. If $M=1$, then $G=V$ is abelian. Otherwise, by
Lemma~\ref{lem:subnormal-fitting}, the normal subgroup $Q=D$ meets
every nontrivial subnormal subgroup of $M$ nontrivially, and it acts
fixed-point-freely by (ii). Lemma~\ref{lem:affine-lifting} gives
$\W(G)=G$ and the description of the subnormal subgroups.
\end{proof}

Thus, for a faithful irreducible module, whether $V\rtimes M$ is a
$\W$-group depends only on $\Fit(M)$.

\subsection{Height bounds}
\label{subsec:height-bounds}
The height bound for primitive $\W$-groups now follows by counting
layers: the module $V$, the Fitting subgroup of $M$, and at most two
further layers in $M/\Fit(M)$.
\begin{theorem}\label{thm:primitive-height}
Let $G=V\rtimes M$ be as in Theorem~\ref{thm:affine-criterion}, and
suppose that $\W(G)=G$. Then $h(G)\le4$, and the following are
equivalent:
\begin{enumerate}
\item $h(G)=4$;
\item $\Fit(M)=C\times Q$, where $C$ is cyclic of odd order and
$Q\cong Q_8$, and the homomorphism
$\theta:M/\Fit(M)\to\Out(Q)\cong S_3$ induced by conjugation is
surjective.
\end{enumerate}
If the Sylow $2$-subgroups of $M$ are abelian, then $\Fit(M)$ is
cyclic, $M/\Fit(M)$ is abelian, and $h(G)\le3$. This applies in
particular if $|G|$ is odd or the Sylow $2$-subgroups of $G$ are
abelian.
\end{theorem}
\begin{proof}
By Theorem~\ref{thm:affine-criterion}, every abelian subgroup of
$\Fit(M)$ is cyclic and $h(G)=1+h(M)$. Proposition~\ref{prop:rank-one}
gives $h(M)\le3$, with equality exactly under (2). If the Sylow
$2$-subgroups of $M$ are abelian, then so is the Sylow $2$-subgroup
$P$ of $\Fit(M)$; the generalized quaternion case is excluded, so
$P$ is cyclic, $\Fit(M)$ is cyclic, and $M/\Fit(M)$ embeds in the
abelian group $\Aut(\Fit(M))$ by Lemma~\ref{lem:embedding}. Hence
$h(M)\le2$. Only the Sylow $2$-subgroup of $\Fit(M)$ enters this
argument: $h(G)\le3$ as soon as it is not isomorphic to $Q_8$. If
$|G|$ is odd or the Sylow $2$-subgroups of $G$ are abelian, then the
Sylow $2$-subgroups of its subgroup $M$ are abelian.
\end{proof}

\begin{corollary}\label{cor:primitive-height-four}
The soluble primitive $\W$-groups of height four are exactly the
groups $V\rtimes M$, where $M$ is a soluble group acting faithfully
and irreducibly on an elementary abelian group $V\ne0$ such that
$\Fit(M)=C\times Q_8$, with $C$ cyclic of odd order, and conjugation
maps $M/\Fit(M)$ onto $\Out(Q_8)\cong S_3$.
\end{corollary}
\begin{proof}
If $M$ has these properties, then every abelian subgroup of $\Fit(M)$
is cyclic, so $V\rtimes M$ is a $\W$-group by
Theorem~\ref{thm:affine-criterion}, and has height four by
Theorem~\ref{thm:primitive-height}. The converse is
Theorem~\ref{thm:primitive-height}.
\end{proof}

The smallest example is $\AGL(2,3)$
(Proposition~\ref{prop:agl}). The conditions concern $M$ alone:
the $S_3$ condition is the surjectivity of the induced outer action
$M/\Fit(M)\longrightarrow\Out(Q_8)\cong S_3$. It neither asserts
nor excludes the existence of a subgroup isomorphic to $S_3$.

\begin{theorem}\label{thm:soluble-height}
If $G$ is soluble and $\W(G)=G$, then $h(G)\le4$.
If, in addition, $|G|$ is odd or the Sylow $2$-subgroups of $G$
are abelian, then $h(G)\le3$. Moreover, $h(G)=4$ if and only if $G$
has a primitive quotient $G/\Core_G(L)$, with $L$ maximal, of the form
described in Corollary~\ref{cor:primitive-height-four}.
\end{theorem}
\begin{proof}
The case $G=1$ is clear. By Lemma~\ref{lem:primitive-quotients}, $h(G)$
is the maximum of the heights of the primitive quotients
$Y=G/\Core_G(L)$, each of the form $V\rtimes J$ with $J$ faithful and
irreducible on $V$. By Theorem~\ref{thm:functor}, each $Y$ is a
$\W$-group, so Theorem~\ref{thm:primitive-height} gives $h(Y)\le4$. If
$G$ has odd order or abelian Sylow $2$-subgroups, so does $Y$, and
$h(Y)\le3$. Finally, $h(G)=4$ exactly when some $Y$ has height four,
and Corollary~\ref{cor:primitive-height-four} describes these $Y$.
\end{proof}

\begin{corollary}\label{cor:soluble-length}
For every soluble group $G$,
\[
 \ellW(G)\le h(G)\le4\ellW(G).
\]
If $|G|$ is odd or the Sylow $2$-subgroups of $G$ are abelian,
then $h(G)\le3\ellW(G)$.
\end{corollary}
\begin{proof}
The lower inequality is Theorem~\ref{thm:general-height}. Set
$n=\ellW(G)$. By Theorem~\ref{thm:structure}(iv), every factor
\[
 X_i=\W_{i+1}(G)/\W_i(G)=\W(G/\W_i(G))\qquad(0\le i<n)
\]
is a soluble $\W$-group, and Theorem~\ref{thm:soluble-height} gives
$h(X_i)\le4$. Subadditivity along the upper series yields
\[
 h(G)\le\sum_{i=0}^{n-1}h(X_i)\le4n.
\]
When $G$ has odd order or abelian Sylow $2$-subgroups, every $X_i$
retains the corresponding property, since each is a subgroup of a
quotient of $G$, and each height is at most three.
\end{proof}

\section{Examples of length one}
\label{sec:construction}
We now show that the bounds of Section~\ref{sec:affine} are attained
by primitive groups, and construct the blocks used in
Section~\ref{sec:extremal}. Lemma~\ref{lem:GL23} records the structure
of $\GL(2,3)$. The group $\AGL(2,3)$ attains height four
(Proposition~\ref{prop:agl}), and a group of order $21\cdot29^3$ attains
height three with odd order (Proposition~\ref{prop:sharp3}).
Theorem~\ref{thm:avoid-primes} provides odd-order blocks of height
three that avoid any prescribed finite set of primes.

\begin{lemma}\label{lem:GL23}
Let $M=\GL(2,3)$, $L=\SL(2,3)$, and let $Q\cong Q_8$ be the inverse
image of the normal Klein four subgroup of $M/Z(M)\cong S_4$. Then
\[
 \Fit(M)=Q,\qquad M/Q\cong S_3,\qquad M'=L,\qquad h(M)=3.
\]
\end{lemma}
\begin{proof}
The action on the four one-dimensional subspaces of $\FF_3^2$ gives
$M/Z\cong S_4$, with $Z=\{\pm I\}$, and $L/Z\cong A_4$. The matrices
\[
 \begin{pmatrix}0&1\\2&0\end{pmatrix},\qquad
 \begin{pmatrix}1&1\\1&2\end{pmatrix}
\]
generate a quaternion subgroup of order eight; its image modulo $Z$
is the unique subgroup of order four in $A_4$. Hence it is precisely
$Q$, which is normal in $M$, and $M/Q\cong S_4/V_4\cong S_3$.

The derived group of $L/Z\cong A_4$ is $Q/Z$, so $L'Z=Q$.
Every proper subgroup of $Q_8$ has image of order at most two modulo
its center. Consequently $L'=Q$. Similarly, $M'Z=L$ because
$(M/Z)'=A_4$. Taking derived groups of $M'Z=L$ gives
$Q=L'=(M')'\le M'$, and therefore $Z\le M'$ and $M'=L$.

The image of $\Fit(M)$ in $S_4$ is a nilpotent normal subgroup, hence
contained in $V_4$, so $\Fit(M)\le Q$; and $Q$ is nilpotent and normal.
Thus $\Fit(M)=Q$ and $h(M)=1+h(S_3)=3$.
\end{proof}

\begin{proposition}\label{prop:agl}
The group $B=\AGL(2,3)=\FF_3^2\rtimes\GL(2,3)$, of order $432$, is a
soluble primitive $\W$-group with
\[
 h(B)=4,\qquad \Fit(B)=\FF_3^2,\qquad B/B'\cong C_2,\qquad
 \Sigma(B)=\{2\}.
\]
It has the least possible order among soluble $\W$-groups of
Fitting height four.
\end{proposition}
Here $\Sigma(B)$ is the set of primes defined before
Theorem~\ref{thm:main-wreath}; see also Lemma~\ref{lem:sigma}.
\begin{proof}
The natural module $V=\FF_3^2$ is faithful and irreducible, since
$\GL(2,3)$ is transitive on nonzero vectors. By Lemma~\ref{lem:GL23},
$\Fit(M)=Q_8$ for $M=\GL(2,3)$, and conjugation maps $M/Q_8\cong S_3$
isomorphically onto $\Out(Q_8)$, because $C_M(Q_8)=Z(M)\le Q_8$.
Corollary~\ref{cor:primitive-height-four} shows that $B$ is a
$\W$-group of height four, and $\Fit(B)=V$ by
Lemma~\ref{lem:affine-fitting}. Since $-I$ acts as $-1$ on $V$, we
have $V=[V,-I]\le B'$, and $B'/V=M'=\SL(2,3)$ by Lemma~\ref{lem:GL23};
hence $B/B'\cong\GL(2,3)/\SL(2,3)\cong C_2$. Lemma~\ref{lem:sigma}
gives $\Sigma(B)\subseteq\pi(B/B')=\{2\}$, and the quotient
$B/VQ_8\cong S_3$ is nonnilpotent with Fitting quotient $C_2$, so
$2\in\Sigma(B)$ by the same lemma.

For minimality, let $G$ be any soluble $\W$-group of height four.
By Theorem~\ref{thm:soluble-height}, it has a primitive quotient
$U\rtimes J$ as in Corollary~\ref{cor:primitive-height-four}.
Thus $|J|\ge8\cdot6=48$. If $U$ has characteristic $p$, then
$O_p(J)=1$ by faithfulness and irreducibility, so $p$ is odd because
$\Fit(J)$ contains a normal $Q_8$. The faithful restriction to
$Q_8$ cannot have dimension one, since $\GL(1,p)$ is abelian.
Consequently $|U|\ge3^2=9$, and $|G|\ge|U\rtimes J|\ge432$.
\end{proof}

\begin{proposition}\label{prop:sharp3}
Let $\zeta\in\FF_{29}^{\times}$ have order seven. On $V=\FF_{29}^3$
put
\[
 a=\operatorname{diag}(\zeta,\zeta^2,\zeta^4),
 \qquad b(e_1)=e_2,\quad b(e_2)=e_3,\quad b(e_3)=e_1,
\]
and $M=\langle a,b\rangle$. Then $M\cong C_7\rtimes C_3$, and
$B=V\rtimes M$, of order $21\cdot29^3$, is a primitive $A$-group of odd
order with
\[
 \W(B)=B,\qquad h(B)=3,\qquad \Fit(B)=V,\qquad B/B'\cong C_3,\qquad
 \Sigma(B)\subseteq\{3\}.
\]
\end{proposition}
\begin{proof}
Such $\zeta$ exists because $7\mid28$. We have $a^7=b^3=1$ and
$b^{-1}ab=a^2$, so $M\cong C_7\rtimes C_3$ is nonabelian with
$\Fit(M)=\langle a\rangle$ cyclic. The eigenvalues
$\zeta,\zeta^2,\zeta^4$ of $a$ are distinct, so the $a$-invariant
subspaces are spanned by subsets of $\{e_1,e_2,e_3\}$; as $b$
permutes the $e_i$ transitively, $V$ is irreducible, and it is
faithful. Theorem~\ref{thm:affine-criterion} gives $\W(B)=B$ and
$h(B)=1+h(M)=3$, and $\Fit(B)=V$. Since $a-1$ is invertible on $V$,
$V=[V,a]\le B'$; also $[a,b]=a^{-1}a^b=a$, so $B'=V\langle a\rangle$
and $B/B'\cong C_3$, whence $\Sigma(B)\subseteq\{3\}$ by
Lemma~\ref{lem:sigma}. The Sylow subgroups of $B$ are $V$ and cyclic
groups of orders $7$ and $3$.
\end{proof}

Propositions~\ref{prop:agl} and~\ref{prop:sharp3} show that the
coefficients $4$ and $3$ in Corollary~\ref{cor:soluble-length} cannot
be lowered, the latter also for soluble $A$-groups without an order
restriction.

\begin{theorem}\label{thm:avoid-primes}
For every finite set $\Pi$ of primes, there is a primitive odd-order
soluble $A$-group $B$ such that
\[
 \pi(B)\cap\Pi=\varnothing,\qquad \W(B)=B,\qquad h(B)=3.
\]
More explicitly, choose distinct odd primes $s,r,t$ outside $\Pi$
with $s\mid r-1$ and $r\mid t-1$. Then one can take
$B=\mathbb F_t^{\,s}\rtimes(C_r\rtimes C_s)$, and $B/B'\cong C_s$.
\end{theorem}
\begin{proof}
Choose an odd prime $s$ outside $\Pi$. Dirichlet's theorem gives a
prime $r\equiv1\pmod{s}$ outside $\Pi\cup\{s\}$, and then a prime
$t\equiv1\pmod r$ outside $\Pi\cup\{s,r\}$.

Let $M=\langle a\rangle\rtimes\langle b\rangle\cong C_r\rtimes C_s$
with faithful action, and put $Q=\langle a\rangle=\Fit(M)$. Since
$r\mid t-1$, choose a faithful one-dimensional $\mathbb F_tQ$-module $Y$
and put $V=\Ind_Q^M Y$, of dimension $s$. Restricted to $Q$, it is the
sum of $s$ pairwise nonisomorphic nontrivial one-dimensional modules,
permuted transitively by $M$; hence $V$ is faithful and irreducible.
Theorem~\ref{thm:affine-criterion} gives $\W(B)=B$ and $h(B)=3$ for
$B=V\rtimes M$. As in Proposition~\ref{prop:sharp3}, $B'=VQ$ and
$B/B'\cong C_s$. The Sylow $t$-subgroup is $V$, and Sylow $r$- and
$s$-subgroups are cyclic, so $B$ is an odd-order $A$-group with
$\pi(B)=\{s,r,t\}$.
\end{proof}

\section{Soluble regular wreath products}
\label{sec:soluble-wreath}
To stack $\W$-groups we use regular wreath products $B\wr E$. We want
to know when the base $B^E$ is exactly $\W(B\wr E)$, because then the
length increases by one and the heights add
(Theorem~\ref{thm:main-wreath}). Section~\ref{subsec:wreath-heights}
computes Fitting subgroups and heights of wreath products, and shows
that $\W(B\wr E)\le B^E$ when $B$ is not nilpotent.
Section~\ref{subsec:sigma} introduces the set of primes $\Sigma(B)$,
which measures the obstruction to equality, and
Section~\ref{subsec:base-criterion} proves the base criterion.

Throughout this section $B$ and $E$ are finite soluble groups and
$B\wr E=B^E\rtimes E$ is the regular wreath product. We write $B_i$
for the coordinate subgroup of $B^E$ at $i\in E$.

\subsection{Fitting subgroups and heights}
\label{subsec:wreath-heights}
\begin{lemma}\label{lem:wreath-fitting}
Let $1\ne D$ and $E$ be finite soluble groups, and set $X=D\wr E$,
$A=D^E$. If $D$ is not a group of prime-power order, then
\[
 \Fit(X)=\Fit(D)^E.
\]
If $D$ is a $p$-group, then
\[
 \Fit(X)=D^E\rtimes O_p(E).
\]
\end{lemma}
\begin{proof}
If $x\in O_r(X)$ has nonidentity top component, then
$[A,x]\le O_r(X)\cap A$, and this is an $r$-group. Since $E$ acts
regularly, $x$ moves every coordinate. For $u\in D$ placed at a
coordinate $i$, the commutator $[u,x]=u^{-1}u^x$ has component
$u^{-1}$ at $i$, because $u^x$ is supported at a different coordinate.
Thus the $r$-group $[A,x]$ projects onto $D$ at every coordinate,
and $D$ is an $r$-group. If $D$ is not of prime-power order, all $O_r(X)$
therefore lie in $A$, giving $\Fit(X)\le\Fit(A)=\Fit(D)^E$.
The reverse inclusion holds because $\Fit(A)$ is characteristic in $A$
and hence nilpotent normal in $X$.

If $D$ is a $p$-group, the full preimage $A\rtimes O_p(E)$ is a
normal $p$-subgroup and equals $O_p(X)$. For $r\ne p$, the normal
subgroups $O_r(X)$ and $A$ centralize each other. Regularity and
$D\ne1$ imply $C_X(A)\le A$: an element moving a coordinate cannot
centralize a nonidentity element supported there. Thus $O_r(X)\le A$
and $O_r(X)=1$. This proves the second formula.
\end{proof}

\begin{theorem}\label{thm:wreath-height}
Let $B,E$ be finite soluble groups with $B\ne1$, put $b=h(B)$, and set
$C=B/\Fit_{b-1}(B)$. Then
\[
 h(B\wr E)=
 \begin{cases}
 b+h(E),& C\text{ is not of prime-power order},\\
 b+h(E/O_p(E)),& C\text{ is a }p\text{-group}.
 \end{cases}
\]
\end{theorem}
\begin{proof}
For $0\le i<b-1$, the group $B/\Fit_i(B)$ has height at least two,
and is not of prime-power order. Repeated application of
Lemma~\ref{lem:wreath-fitting}, together with the natural
isomorphisms $(B\wr E)/\Fit_i(B)^E\cong(B/\Fit_i(B))\wr E$, gives
\[
 \Fit_i(B\wr E)=\Fit_i(B)^E\qquad(0\le i\le b-1).
\]
The remaining quotient is $C\wr E$, with $C$ nontrivial nilpotent.
If $C$ is not of prime-power order, its Fitting subgroup is the base
and its quotient by that base is $E$, giving height $1+h(E)$.
If $C$ is a $p$-group, its Fitting quotient is $E/O_p(E)$, giving
height $1+h(E/O_p(E))$. Adding the first $b-1$ layers proves the result.
\end{proof}

\begin{lemma}\label{lem:base-upper}
Let $B$ be a nonnilpotent group, $E$ any group, $G=B\wr E$ and
$N=B^E$. Then $N_G(B_i^{\N})=N$ for every $i\in E$ when $E\ne1$, and
\[
 \W(G)\le N,\qquad \ellW(G)\ge1+\ellW(E).
\]
\end{lemma}
\begin{proof}
If $E=1$, then $G=N=B\ne1$ and both assertions are clear. Let $E\ne1$.
Each $B_i$ is normal in $N$ and hence subnormal in $G$, and
$B_i^{\N}\ne1$ is normalized by $N$. An element $xe$, with $x\in N$
and $1\ne e\in E$, maps $B_i^{\N}$ to $B_{ie}^{\N}$, which is supported
at a different coordinate. Hence $N_G(B_i^{\N})=N$ and $\W(G)\le N$.
Thus $G/\W(G)$ maps onto $G/N\cong E$, and
Lemma~\ref{lem:series-shift} gives
$\ellW(G)=1+\ellW(G/\W(G))\ge1+\ellW(E)$.
\end{proof}

\subsection{The prime obstruction in the base}
\label{subsec:sigma}
The obstruction to $\W(B\wr E)=B^E$ comes from the primes $r$ for which
$B$ has a nonnilpotent quotient whose Fitting quotient is an $r$-group
(Lemma~\ref{lem:sigma}). We collect them in a set. Define
\[
\Sigma(B)=\left\{r:\begin{array}{l}
B\text{ has a quotient }V\rtimes P,\ V\ne1\text{ elementary abelian},\\
1\ne P\text{ an }r\text{-group acting faithfully and irreducibly on }V
\end{array}\right\}.
\]
The characteristic of $V$ in this definition differs from $r$, since a
nontrivial $r$-group has only trivial irreducible modules in
characteristic $r$.

\begin{lemma}\label{lem:sigma}
For a finite soluble group $B$ and a prime $r$, one has $r\in\Sigma(B)$
if and only if there is $D\unlhd B$ such that $Q=B/D$ is nonnilpotent
and $Q/\Fit(Q)$ is an $r$-group. Moreover,
\[
 \Sigma(B)\subseteq\pi(B/B').
\]
\end{lemma}
\begin{proof}
If $Q=V\rtimes P$ is as in the definition of $\Sigma(B)$, then
$\Fit(Q)=V$ by Lemma~\ref{lem:affine-fitting}, so $Q$ is nonnilpotent
and $Q/\Fit(Q)\cong P$ is an $r$-group. This proves the forward
implication.

Conversely, a finite nonnilpotent group $Q$ has a nonnormal maximal
subgroup $M$, since a finite group whose maximal subgroups are all
normal is nilpotent. Set $L=\Core_Q(M)$. By
Lemma~\ref{lem:primitive-quotients}, $Q/L=V\rtimes P$ with $P=M/L$
faithful and irreducible on $V$, and $P\ne1$ because otherwise $M=L$
would be normal. The image of $\Fit(Q)$ lies in $\Fit(Q/L)=V$, so $P$
is a quotient of the $r$-group $Q/\Fit(Q)$. Thus $r\in\Sigma(B)$.

Finally, a nontrivial $r$-group $P$ has a quotient of order $r$.
Since $B$ maps onto $P$, it has an abelian quotient of order $r$.
This proves the displayed containment.
\end{proof}

\subsection{The base criterion}
\label{subsec:base-criterion}
\begin{theorem}\label{thm:wreath-base}
Let $B,E$ be finite soluble groups such that $\W(B)=B$ and $\res B\ne1$.
Put $G=B\wr E$ and $N=B^E$. Then the following are equivalent:
\begin{enumerate}
\item $\W(G)=N$;
\item $\Sigma(B)\cap\pi(\Fit(E))=\varnothing$;
\item $O_r(E)=1$ for every $r\in\Sigma(B)$.
\end{enumerate}
Whenever they hold,
\[
 \ellW(G)=1+\ellW(E),\qquad h(G)=h(B)+h(E).
\]
\end{theorem}
The idea of the proof is as follows. Suppose that $r\in\Sigma(B)$ and
that $E$ has an element of order $r$ in $O_r(E)$. This element
permutes the coordinates, and the residual of a suitable subnormal
subgroup consists of tuples whose coordinate sums vanish along its
orbits. Changing a single coordinate destroys this property, so the
base does not normalize that residual. Conversely, if no prime of
$\Sigma(B)$ divides $|\Fit(E)|$, then for a subnormal subgroup $H$ not
contained in the base, the action of the $r$-part of its image on the
coordinates forces every coordinate quotient involved to be nilpotent.
This places $\res N$ inside $\res H$, and the base, being nilpotent
modulo $\res N$, normalizes $\res H$.

\begin{proof}
The equivalence of (2) and (3) follows from the Sylow decomposition of
the nilpotent normal subgroup $\Fit(E)$: its Sylow $r$-subgroup equals
$O_r(E)$.

\medskip\noindent\emph{(1) implies (3).}
Suppose $\W(G)=N$ and, towards a contradiction, choose
$r\in\Sigma(B)$ with $O_r(E)\ne1$. Let $\overline B=V\rtimes P$ be
the corresponding quotient of $B$, with $V$ elementary abelian of
characteristic $q\ne r$. Applying this quotient coordinatewise gives
\[
 \overline G=\overline B\wr E=U\rtimes T,
 \qquad U=V^E,\quad T=P^E\rtimes E.
\]
Theorem~\ref{thm:functor} implies
\begin{equation}\label{eq:wreath-quotient}
 \overline B^E\le \W(\overline G).
\end{equation}
Take $s\in O_r(E)$ of order $r$ and set $A=\langle s\rangle$.
The group $P^E\rtimes O_r(E)$ is a normal $r$-subgroup of $T$.
Every subgroup of a finite $r$-group is subnormal, hence $A\sn T$
and $H=U\rtimes A\sn\overline G$. We claim that
\[
 \res H=[U,A].
\]
Indeed, $H/[U,A]$ is abelian. Moreover $[U,A]=(s-1)U$, and since
$q\ne r$, Maschke's theorem gives
$U=\ker(s-1)\oplus(s-1)U$, so that $s-1$ is invertible on $(s-1)U$.
Thus $[U,A,A]=[U,A]$, and the lower central series of $H$ stabilizes
at $[U,A]$.

On each orbit of $s$ on $E$, the subgroup $[U,A]$ consists of tuples
whose coordinate sum is zero. Choose distinct coordinates $i,j$ in
one orbit, and $a\in P$, $v\in V$ with $v^a\ne v$. The tuple with
entries $v$ at $i$, $-v$ at $j$, and zero elsewhere belongs to
$[U,A]$. Acting by $a$ only at coordinate $i$ changes its orbit sum
to $v^a-v\ne0$. Hence this element of $P^E$ does not normalize
$\res H$, contradicting \eqref{eq:wreath-quotient}.

\medskip\noindent\emph{(2) implies (1).}
Assume (2). By Lemma~\ref{lem:base-upper}, $\W(G)\le N$, and
$\W(N)=N$ by Theorem~\ref{thm:W-product}. It remains to show that $N$
normalizes $R=\res H$ for each $H\sn G$. If $H\le N$, this follows
from $\W(N)=N$.

Otherwise let $1\ne T\sn E$ be the image of $H$. The nontrivial
soluble group $T$ has $\Fit(T)\ne1$. Choose $r$ with
$S=O_r(T)\ne1$. By Lemma~\ref{lem:subnormal-fitting}, $\Fit(T)$ lies
in $\Fit(E)$; therefore $r\notin\Sigma(B)$.
Let $H_0$ be the preimage of $S$ in $H$, and put
\[
 K=H_0\cap N,\qquad D=\res K.
\]
Here $H_0\unlhd H$, $K\sn N$, and $K\unlhd H_0$.
Since $\W(N)=N$, the subgroup $D$ is normal in $N$; it is also
characteristic in $K$ and hence normal in $H_0$. Thus $D\unlhd NH_0$.
Lemma~\ref{lem:residual-product}, iterated over the coordinates, gives
\[
 D=\prod_iD_i,\qquad D_i\unlhd B_i.
\]
The group $K/D$ is nilpotent and subnormal in $N/D$, so it lies in
$\Fit(N/D)$. Let $F_0$ be the preimage of this Fitting subgroup in $N$.
Then $K\le F_0\unlhd NH_0$ and
\[
 F_0=\prod_iF_i,\qquad F_i/D_i=\Fit(B_i/D_i).
\]
In $Y=NH_0/F_0$, the image $\overline H_0$ is a subnormal $r$-group:
it is a quotient of $H_0/K\cong S$. Lemma~\ref{lem:subnormal-fitting}
gives $\overline H_0\le O_r(Y)$, and
\[
 [N/F_0,\overline H_0]\le O_r(Y)\cap(N/F_0)
\]
is an $r$-group.

Choose $x\in H_0$ with nonidentity image in $S$. For each coordinate
$i$, and $u\in B_i$, use the convention $[u,x]=u^{-1}u^x$.
The two supported factors occur at distinct coordinates. The projection
of $[u,x]$ to coordinate $i$ is $u^{-1}$, since $x$ moves $i$.
Thus the $r$-group $[N/F_0,\overline H_0]$ projects onto every
$B_i/F_i$. Each
\[
 (B_i/D_i)/\Fit(B_i/D_i)
\]
is consequently an $r$-group. Lemma~\ref{lem:sigma}, together with
$r\notin\Sigma(B)$, forces $B_i/D_i$ to be nilpotent for every $i$.
It follows that $N/D$ is nilpotent and
\begin{equation}\label{eq:wreath-contain}
 \res N\le D\le\res{H_0}\le\res H=R.
\end{equation}
In $G/\res N$, the subgroup $N/\res N$ is nilpotent and normal.
By Lemma~\ref{lem:Fstar}, it normalizes the nilpotent residual of
the image of $H$, namely $R/\res N$. In view of \eqref{eq:wreath-contain},
this says exactly that $N$ normalizes $R$. Thus $N\le \W(G)$.

\medskip\noindent\emph{Length and height.}
Under the equivalent conditions, $\W(G)=N$ and $G/N\cong E$, so
Lemma~\ref{lem:series-shift} gives $\ellW(G)=1+\ellW(E)$. For height,
$b=h(B)\ge2$. If $C=B/\Fit_{b-1}(B)$ is not of prime-power order,
apply Theorem~\ref{thm:wreath-height}. Otherwise $C$ is a $p$-group.
The group $Q=B/\Fit_{b-2}(B)$ has height two, and
$Q/\Fit(Q)\cong C$. Lemma~\ref{lem:sigma} implies $p\in\Sigma(B)$,
so $O_p(E)=1$, and Theorem~\ref{thm:wreath-height} again yields
$h(G)=h(B)+h(E)$.
\end{proof}

\begin{corollary}\label{cor:wreath-coprime}
Let $B,E$ be soluble with $\W(B)=B$ and $B^{\N}\ne1$, and suppose that
\[
 \gcd(|B/B'|,|\Fit(E)|)=1.
\]
Then, for $G=B\wr E$,
\[
\W(G)=B^E,\qquad \ellW(G)=1+\ellW(E),\qquad h(G)=h(B)+h(E),
\qquad \Fit(G)=\Fit(B)^E.
\]
\end{corollary}
\begin{proof}
By Lemma~\ref{lem:sigma}, $\Sigma(B)\subseteq\pi(B/B')$, so
Theorem~\ref{thm:wreath-base} gives the first three identities. Since
$B$ is not nilpotent, it is not of prime-power order, and
Lemma~\ref{lem:wreath-fitting} gives $\Fit(G)=\Fit(B)^E$.
\end{proof}

\subsection{Examples}
\begin{example}[A base admitting every soluble top group]
\label{ex:c7c6}
Let $B=C_7\rtimes C_6$, with the faithful action of $C_6$ on $C_7$.
Then $\W(B)=B$, $\res B=C_7$, $\Fit(B)=C_7$, and
$\Sigma(B)=\varnothing$. Consequently, for every finite soluble group
$E$,
\[
 \W(B\wr E)=B^E,\qquad
 \ellW(B\wr E)=1+\ellW(E),\qquad h(B\wr E)=2+h(E).
\]
\end{example}
\begin{proof}
The group $B$ is primitive on $C_7$ with complement $M=C_6$.
Since $\Fit(M)=C_6$ is cyclic, $\W(B)=B$ by
Theorem~\ref{thm:affine-criterion}. Also $\Fit(B)=C_7$ by
Lemma~\ref{lem:affine-fitting}, and
$\res B=C_7$ because $[C_7,C_6]=C_7$ and $B/C_7$ is abelian.
Every nontrivial normal subgroup $D$ of $B$ contains $C_7$: otherwise
$D\cap C_7=1$, whence $[D,C_7]=1$ and $D\le C_B(C_7)=C_7$, a
contradiction. Every proper quotient of $B$ is therefore abelian. But
$B/\Fit(B)=C_6$ is not an $r$-group for any prime $r$, so
Lemma~\ref{lem:sigma} gives $\Sigma(B)=\varnothing$. The formulas
follow from Theorem~\ref{thm:wreath-base}. For $E=C_2$ or $C_3$ the
condition $\gcd(|B/B'|,|\Fit(E)|)=1$ fails.
\end{proof}

\begin{example}[The obstruction for $S_3$]
\label{ex:s3}
For $B=S_3$ one has $\W(B)=B$ and $\Sigma(B)=\{2\}$. Thus, for every
finite soluble $E$,
\[
 \W(S_3\wr E)=S_3^E\quad\Longleftrightarrow\quad O_2(E)=1.
\]
For $G=S_3\wr C_2$ the equality fails: $h(G)=2$, $\ellW(G)=2$, and
$\W(G)$ is the subgroup of index four of $G$ consisting of the pairs
$(x,y)\in S_3\times S_3$ with $x$ and $y$ of the same sign.
\end{example}
\begin{proof}
The subnormal subgroups of $S_3$ are $1,C_3,S_3$, whose residuals
are $1,1,C_3$. The only nonnilpotent quotient of $S_3$ is itself, with
Fitting quotient $C_2$. The equivalence follows from
Theorem~\ref{thm:wreath-base}, and $h(G)=2+h(C_2/O_2(C_2))=2$ by
Theorem~\ref{thm:wreath-height}.

Put $A=C_3\times C_3$ and let $X$ be the subgroup of $S_3\times S_3$
consisting of pairs of the same sign. Since $G/A\cong D_8$ is
nilpotent, Lemma~\ref{lem:residual-basics} gives $H^{\N}\le A$
for every $H\le G$. On the vector space $A\cong\FF_3^2$, each
element of $X$ acts as one of the scalars $1,-1$; thus $X$ normalizes
every subgroup of $A$, and $X\le\W(G)$.
Let $t$ interchange the two coordinates. The subgroup
$H=A\langle t\rangle$ is subnormal in $G$, because $H/A$ is a
subgroup of the $2$-group $G/A$. Its residual is
$[A,t]=\{(a,-a):a\in\FF_3\}$: this subgroup has cyclic quotient
in $H$ and satisfies $[[A,t],t]=[A,t]$.
Its normalizer in $S_3\times S_3$ is exactly $X$. Since
$\W(G)\le S_3\times S_3$ by Lemma~\ref{lem:base-upper}, it follows
that $\W(G)=X$ and $|G:\W(G)|=4$. Finally, $G/X$ is a nontrivial
nilpotent group, so $\ellW(G)=2$.
\end{proof}

\section{Extremal heights and all pairs at a given length}
\label{sec:extremal}
For $n\ge1$ put
\begin{align*}
 a(n)&=\max\{h(G):G\text{ finite soluble},\ \ellW(G)=n\},\\
 a_{\mathrm{odd}}(n)&=\max\{h(G):G\text{ finite soluble of odd order},\ \ellW(G)=n\},\\
 a_A(n)&=\max\{h(G):G\text{ a finite soluble }A\text{-group},\ \ellW(G)=n\}.
\end{align*}
By Corollary~\ref{cor:soluble-length}, $a(n)\le4n$ and
$a_{\mathrm{odd}}(n),a_A(n)\le3n$. In this section we assemble the
blocks of Section~\ref{sec:construction} by the wreath products of
Section~\ref{sec:soluble-wreath}. Iterating $\AGL(2,3)$ gives height $4n$
at length $n$ (Theorem~\ref{thm:soluble-exact}), and the odd-order
blocks give $3n$ (Theorem~\ref{thm:odd-exact}). Mixing blocks of heights
two, three and four with groups of equal height and length realizes
every intermediate pair (Theorem~\ref{thm:all-pairs}).
Table~\ref{tab:examples} lists the examples. Finally,
Section~\ref{subsec:equality} describes the groups attaining
$h(G)=4\ellW(G)$.

\begin{theorem}\label{thm:soluble-exact}
Let $G_1=\AGL(2,3)$ and $G_{n+1}=\AGL(2,3)\wr G_n$. Then $G_n$ is a
$\{2,3\}$-group with
\[
 \ellW(G_n)=n,\qquad h(G_n)=4n,\qquad \Fit(G_n)\text{ a }3\text{-group}.
\]
In particular $a(n)=4n$ for every $n\ge1$.
\end{theorem}
\begin{proof}
By Proposition~\ref{prop:agl}, $B=\AGL(2,3)$ is a $\W$-group with
$B^{\N}\ne1$, $B/B'\cong C_2$ and $\Fit(B)$ a $3$-group; this is the
case $n=1$. If the assertion holds for $G_n$, then
$\gcd(|B/B'|,|\Fit(G_n)|)=1$, and
Corollary~\ref{cor:wreath-coprime} gives $\ellW(G_{n+1})=n+1$,
$h(G_{n+1})=4+4n$ and $\Fit(G_{n+1})=\Fit(B)^{G_n}$, a $3$-group.
\end{proof}

\begin{theorem}\label{thm:odd-exact}
For every $n\ge1$, $a_{\mathrm{odd}}(n)=a_A(n)=3n$. More precisely:
\begin{enumerate}
\renewcommand{\labelenumi}{\textup{(\roman{enumi})}}
\item if $B$ is the group of Proposition~\ref{prop:sharp3}, the
iterated wreath powers $H_1=B$, $H_{n+1}=B\wr H_n$ are
$\{3,7,29\}$-groups of odd order with $\ellW(H_n)=n$ and $h(H_n)=3n$;
\item there are odd-order soluble $A$-groups $K_n$ with
$\ellW(K_n)=n$ and $h(K_n)=3n$.
\end{enumerate}
\end{theorem}
\begin{proof}
(i) Since $B/B'\cong C_3$ and $\Fit(B)$ is a $29$-group, the argument
of Theorem~\ref{thm:soluble-exact} applies, with $\Fit(H_n)$ a
$29$-group at every step.

(ii) Take $K_1$ to be a group from Theorem~\ref{thm:avoid-primes}.
Having constructed $K_n$, use that theorem with $\Pi=\pi(K_n)$ to
choose a block $B_{n+1}$ of coprime order, and put
$K_{n+1}=B_{n+1}\wr K_n$. Corollary~\ref{cor:wreath-coprime} gives
$\ellW(K_{n+1})=n+1$ and $h(K_{n+1})=3n+3$. Odd order is preserved.
Since base and top have disjoint prime supports, a Sylow subgroup of
$K_{n+1}$ is either a direct product of Sylow subgroups of the base
or a Sylow subgroup of the top group. Both are abelian by induction,
so $K_{n+1}$ is an $A$-group.
\end{proof}
In (i) the Sylow $3$-subgroups of $H_n$ are not abelian for $n\ge2$;
the disjoint supports in (ii) are needed only for the $A$-group
property.

\begin{lemma}\label{lem:equal-length}
For every $m\ge1$ and $p\in\{2,3\}$ there is a soluble $\{2,3\}$-group
$K$ with $\ellW(K)=h(K)=m$ and $\Fit(K)$ a nontrivial $p$-group.
\end{lemma}
\begin{proof}
For $m=1$ take $K=C_p$. Suppose the assertion holds for $m$ and both
primes. To obtain $p=3$ at $m+1$, take $K$ for $m$ with $\Fit(K)$ a
$2$-group and put $K^+=S_3\wr K$. Since $S_3/\Fit(S_3)\cong C_2$ and
$O_2(K)=\Fit(K)$, Theorem~\ref{thm:wreath-height} gives
$h(K^+)=2+h(K/\Fit(K))=m+1$, and Lemma~\ref{lem:wreath-fitting} gives
$\Fit(K^+)=C_3^K$. To obtain $p=2$, take $K$ with $\Fit(K)$ a
$3$-group and put $K^+=A_4\wr K$; now $A_4/\Fit(A_4)\cong C_3$, so
again $h(K^+)=m+1$, and $\Fit(K^+)=V_4^K$. In both cases
Lemma~\ref{lem:base-upper} and Theorem~\ref{thm:general-height} give
$m+1\le\ellW(K^+)\le h(K^+)=m+1$.
\end{proof}

\begin{theorem}\label{thm:all-pairs}
Let $n\ge1$. A soluble group $G$ with $\ellW(G)=n$ satisfies
$n\le h(G)\le4n$, and for every integer $h$ with $n\le h\le4n$ there
is a soluble $\{2,3,7,29\}$-group $G$ with $\ellW(G)=n$ and $h(G)=h$.
\end{theorem}
The idea is to write $h-n=3a+b$ with $0\le b\le2$. We first build a
group $S$ of length $n-a$ and height $n-a+b$, from groups of equal
length and height and at most one block of height $b+1$. We then wrap
it $a$ times in $\AGL(2,3)$; each wrapping raises the length by one and
the height by four.
\begin{proof}
The inequalities are Corollary~\ref{cor:soluble-length}. We use the
$\W$-groups
\[
 B_2=C_7\rtimes C_6,\qquad B_3=\FF_{29}^3\rtimes(C_7\rtimes C_3),\qquad
 B_4=\AGL(2,3)
\]
of Example~\ref{ex:c7c6} and Propositions~\ref{prop:sharp3}
and~\ref{prop:agl}. They are nonnilpotent, $h(B_c)=c$, and
$\Sigma(B_2)=\varnothing$, $\Sigma(B_3)\subseteq\{3\}$,
$\Sigma(B_4)=\{2\}$; their Fitting subgroups are a $7$-group, a
$29$-group and a $3$-group, respectively.

Put $x=h-n\in[0,3n]$ and write $x=3a+b$ with $0\le b\le2$. If $x=3n$,
Theorem~\ref{thm:soluble-exact} applies. Otherwise $3a\le x<3n$, so
$a\le n-1$. We first
construct a soluble group $S$ with $\Fit(S)$ of odd order, as follows.
\begin{itemize}
\item If $b=0$, put $m=n-a\ge1$ and let $S$ be a group of
Lemma~\ref{lem:equal-length} with $\ellW(S)=h(S)=m$ and $\Fit(S)$ a
$3$-group.
\item If $b>0$ and $a<n-1$, put $m=n-a-1\ge1$, let $K$ be a group of
Lemma~\ref{lem:equal-length} with $\ellW(K)=h(K)=m$ and $\Fit(K)$ a
$2$-group, and put $S=B_{b+1}\wr K$. Since
$\Sigma(B_{b+1})\cap\{2\}=\varnothing$,
Theorem~\ref{thm:wreath-base} gives $\ellW(S)=m+1$ and
$h(S)=m+b+1$, and $\Fit(S)=\Fit(B_{b+1})^K$ has odd order by
Lemma~\ref{lem:wreath-fitting}.
\item If $b>0$ and $a=n-1$, put $S=B_{b+1}$, so that $\ellW(S)=1$
and $h(S)=b+1$.
\end{itemize}
In each case $\ellW(S)=n-a$ and $h(S)=n-a+b$. Now put $S_0=S$ and
$S_{k+1}=B_4\wr S_k$. Since $\Fit(S_k)$ has odd order and
$\Sigma(B_4)=\{2\}$, Theorem~\ref{thm:wreath-base} and
Lemma~\ref{lem:wreath-fitting} show inductively that
$\ellW(S_k)=\ellW(S)+k$, $h(S_k)=h(S)+4k$, and that $\Fit(S_k)$ is a
$3$-group for $k\ge1$. The group $G=S_a$ has $\ellW(G)=n$ and
$h(G)=n-a+b+4a=n+x=h$. All groups involved are
$\{2,3,7,29\}$-groups.
\end{proof}

\begin{table}[ht]
\centering
\renewcommand{\arraystretch}{1.2}
\begin{tabular}{llccl}
\hline
Group & Reference & $\ellW$ & $h$ & Primes\\
\hline
$C_7\rtimes C_6$ & Ex.~\ref{ex:c7c6} & $1$ & $2$ & $2,3,7$\\
$\FF_{29}^3\rtimes(C_7\rtimes C_3)$ & Prop.~\ref{prop:sharp3} & $1$ & $3$ & $3,7,29$\\
$\AGL(2,3)$ & Prop.~\ref{prop:agl} & $1$ & $4$ & $2,3$\\
$S_3\wr C_2$ & Ex.~\ref{ex:s3} & $2$ & $2$ & $2,3$\\
$\AGL(2,3)\wr\AGL(2,3)$ & Thm.~\ref{thm:soluble-exact} & $2$ & $8$ & $2,3$\\
$G_n$ (iterated $\AGL(2,3)$) & Thm.~\ref{thm:soluble-exact} & $n$ & $4n$ & $2,3$\\
$H_n$ (iterated, odd order) & Thm.~\ref{thm:odd-exact}(i) & $n$ & $3n$ & $3,7,29$\\
$K_n$ (odd-order $A$-groups) & Thm.~\ref{thm:odd-exact}(ii) & $n$ & $3n$ & $3n$ primes\\
any $h\in[n,4n]$ & Thm.~\ref{thm:all-pairs} & $n$ & $h$ & $\subseteq\{2,3,7,29\}$\\
\hline
\end{tabular}
\caption{Examples. The first three rows are primitive groups of
length one. The groups $K_n$ give odd-order $A$-group examples at
every length; all groups in the table are soluble.}
\label{tab:examples}
\end{table}

\subsection{Equality cases}
\label{subsec:equality}
Finally we describe the groups with $h(G)=4\ellW(G)$. In such a group
every layer of the upper $\W$-series has height exactly four, and the
equality can be detected on a single characteristic subgroup, the
iterated residual $L_{4n-4}(G)$ defined below.
For a soluble group $G$ put $L_0(G)=G$ and $L_{j+1}(G)=L_j(G)^{\N}$.
If $G\ne1$, then $G/G^{\N}$ is nilpotent and
$G^{\N}\le\Fit_{h(G)-1}(G)$, so $h(G^{\N})=h(G)-1$; hence
\begin{equation}\label{eq:iterated-height}
 h(L_j(G))=\max\{h(G)-j,0\}\qquad(j\ge0),
\end{equation}
and the subgroups $L_j(G)$ are characteristic and commute with
epimorphisms by Lemma~\ref{lem:residual-basics}(ii).

\begin{proposition}\label{prop:equality}
Let $c=4$, or let $c=3$ and assume that $|G|$ is odd or the Sylow
$2$-subgroups of $G$ are abelian. Let $G$ be soluble with
$\ellW(G)=n\ge1$ and $h(G)=cn$, and put $N_i=\W_i(G)$. Then
\begin{align*}
 h(N_{i+1}/N_i)&=c &&(0\le i<n),\\
 h(N_i)&=ci,\qquad h(G/N_i)=c(n-i) &&(0\le i\le n).
\end{align*}
Moreover, a soluble group $G$ satisfies $\ellW(G)=n$ and $h(G)=4n$ if
and only if $\W_n(G)=G$ and $h(L_{4n-4}(G))=4$. In this case
$L_{4n-4}(G)\le\W(G)$ is a $\W$-group of height four, so it has a
primitive quotient as in Corollary~\ref{cor:primitive-height-four}.
\end{proposition}
\begin{proof}
The layers $N_{i+1}/N_i$ are $\W$-groups in the relevant class, so
$h(N_{i+1}/N_i)\le c$ by Theorem~\ref{thm:soluble-height}, and
$cn=h(G)\le\sum_ih(N_{i+1}/N_i)\le cn$ forces equality throughout.
Subadditivity gives $h(N_i)\le ci$, and
Lemma~\ref{lem:series-shift} with Corollary~\ref{cor:soluble-length}
gives $h(G/N_i)\le c(n-i)$; since
$cn\le h(N_i)+h(G/N_i)$, both are equalities.

For the second assertion, assume $\W_n(G)=G$ and put $T=L_{4n-4}(G)$.
The quotient $G/\W(G)$ has length at most $n-1$, hence height at most
$4n-4$, so the image of $T$ in it is trivial by
\eqref{eq:iterated-height}; thus $T\le\W(G)$. As a normal subgroup
of $\W(G)$, $T$ is a $\W$-group by Theorem~\ref{thm:functor}. If
$\ellW(G)=n$ and $h(G)=4n$, then $h(T)=4$ by
\eqref{eq:iterated-height}. Conversely, if $h(T)=4$, then
\eqref{eq:iterated-height} gives $h(G)=4n$, and
$4n=h(G)\le4\ellW(G)\le4n$ gives $\ellW(G)=n$. The last assertion is
Theorem~\ref{thm:soluble-height}.
\end{proof}

\appendix
\section{Trivial soluble radical and nonsoluble length}
\label{sec:radical-free}
Throughout this appendix $\Rad(G)$ denotes the soluble radical. We
first compute $\W(G)$ when $\Rad(G)=1$; then we bound the nonsoluble
length. The formula~\eqref{eq:radical-free} does not depend on the
classification of finite simple groups, but the solubility statements
below use the Schreier conjecture, which holds by the classification:
$\Out(S)$ is soluble for every nonabelian finite simple group $S$
(see, e.g., \cite{Wilson}).

\subsection{The formula for $\mathrm W(G)$}
\begin{lemma}\label{lem:socle}
Let $T=S_1\times\cdots\times S_t$, where every $S_i$ is nonabelian
simple.
\begin{enumerate}
\renewcommand{\labelenumi}{\textup{(\roman{enumi})}}
\item Every normal subgroup of $T$ is the product of some of the
$S_i$. In particular, automorphisms of $T$ permute the $S_i$.
\item If $T\le G\le\Aut(T)$, where $T$ is identified with $\Inn(T)$,
then $C_G(T)=1$.
\item If $T\trianglelefteq G$ and $C_G(T)=1$, then $\Rad(G)=1$ and
$\Soc(G)=T$. The minimal normal subgroups of $G$ are the products
$T_{\mathcal O}$ of the factors in the orbits $\mathcal O$ of $G$ on
$\{S_1,\ldots,S_t\}$.
\end{enumerate}
\end{lemma}
\begin{proof}
(i) Let $M\trianglelefteq T$. If the projection of $M$ onto $S_i$ is
nontrivial, it is all of $S_i$. Then $[M,S_i]\leq M\cap S_i$ has
projection $[S_i,S_i]=S_i$, so $S_i\leq M$. Applying this to all
nontrivial projections proves the first assertion. The $S_i$ are then
precisely the minimal normal subgroups of $T$, and so are permuted by
automorphisms.

(ii) An automorphism $\alpha$ of $T$ centralizing $\Inn(T)$ satisfies
$c_{\alpha(x)}=\alpha c_x\alpha^{-1}=c_x$ for every $x\in T$, whence
$\alpha(x)x^{-1}\in Z(T)=1$.

(iii) Let $N\trianglelefteq G$ be soluble. By (i), the soluble normal
subgroup $N\cap T$ of $T$ is trivial, so $[N,T]\le N\cap T=1$ and
$N\le C_G(T)=1$. Thus $\Rad(G)=1$. By (i), $G$ permutes the $S_i$, and
each $T_{\mathcal O}$ is normal in $G$. A normal subgroup of $G$
contained in $T_{\mathcal O}$ is, by (i), a $G$-invariant product of
factors, hence $1$ or $T_{\mathcal O}$; so $T_{\mathcal O}$ is minimal
normal. If $L$ is a minimal normal subgroup of $G$, then $L\cap T\ne1$,
since otherwise $[L,T]=1$. Hence $L\le T$, and $L$ is a $G$-invariant
product of factors, that is, some $T_{\mathcal O}$. Consequently
$\Soc(G)=\prod_{\mathcal O}T_{\mathcal O}=T$.
\end{proof}

\begin{lemma}\label{lem:support}
Suppose $T\trianglelefteq G$ and
\[
 T=S_1\times\cdots\times S_t,
\]
where every $S_i$ is nonabelian simple. If $H\sn G$ and
\[
 I=\{i:S_i\leq H\},\qquad A=\prod_{i\in I}S_i,
 \qquad B=\prod_{i\notin I}S_i,
\]
with an empty product interpreted as $1$, then
\[
 H\cap T=A,\qquad [H,B]=1,\qquad A\leq H^{\N}.
\]
Moreover, writing $L=H^{\N}$, one has the reconstruction identity
\begin{equation}\label{eq:reconstruct}
 L=LT\cap C_G(B).
\end{equation}
\end{lemma}
\begin{proof}
We use Lemma~\ref{lem:socle}(i) repeatedly: normal subgroups of a
direct product of nonabelian simple groups are products of factors,
and automorphisms permute the factors.

Choose a subnormal chain from $H$ to $G$. We prove down this chain
that each term $J$ satisfies: $J\cap T$ is a product of factors, and
$J$ centralizes the product of the remaining factors. This holds for
$J=G$. Suppose it holds for $J$ and $U\trianglelefteq J$. Put
$A_J=J\cap T$. Since $U\cap A_J\trianglelefteq A_J$, the preceding
observation gives $U\cap T=U\cap A_J=A_U$, a product of some of the
factors in $A_J$. Also
\[
 [A_J,U]\leq A_J\cap U=A_U.
\]
Conjugation by $u\in U$ permutes the factors of $A_J$ and induces the
identity on $A_J/A_U$. It therefore fixes every factor $S_i$ of $A_J$
outside $A_U$ setwise. For $s\in S_i$, we have
$s^{-1}s^u\in S_i\cap A_U=1$, so $U$ centralizes $S_i$. By the
inductive hypothesis, $U$ also centralizes every factor outside $A_J$.
This proves the inductive step and hence the first two assertions for
$H$.

The group $A$ is perfect. Its image in the nilpotent group
$H/H^{\N}$ is both perfect and nilpotent, and is therefore trivial.
Thus $A\leq L$. Since $T=A\times B$, we obtain $LT=LB$. Moreover,
$[L,B]=1$ and $Z(B)=1$. An element $lb\in LB$, with $l\in L$ and
$b\in B$, centralizes $B$ if and only if $b\in Z(B)=1$. This
proves~\eqref{eq:reconstruct}.
\end{proof}

\begin{theorem}\label{thm:radical-free}
Let $G\ne1$ be a finite group with $\Rad(G)=1$. Write
\[
 T=\Soc(G)=S_1\times\cdots\times S_t,
\]
where the $S_i$ are nonabelian simple groups. Let
\[
 K=\bigcap_{i=1}^t N_G(S_i),\qquad Q=G/T,
 \qquad D=K/T,
\]
and let $\pi:G\longrightarrow Q$ be the quotient map. Then
\begin{equation}\label{eq:radical-free}
 \W(G)=K\cap\pi^{-1}(\W(Q)).
\end{equation}
In particular, $T\leq \W(G)$,
\begin{equation}\label{eq:modsoc}
\W(G)/T=D\cap \W(Q),
\end{equation}
and $|G:\W(G)|=|G:K|\,|D:D\cap\W(Q)|$. Furthermore, $C_G(T)=1$, $K$
is the kernel of the permutation action of $G$ on the simple direct
factors of $T$, and there is an embedding
\begin{equation}\label{eq:outer}
 D\hookrightarrow\prod_{i=1}^t\Out(S_i).
\end{equation}
\end{theorem}
\begin{proof}
A minimal normal subgroup of a finite group is a direct product of
isomorphic simple groups. An abelian minimal normal subgroup would be
contained in $\Rad(G)$. Thus the socle has the displayed form. Its
simple factors are permuted by $G$, and the kernel of this action is
exactly $K$. In particular, $K\trianglelefteq G$ and $T\leq K$.

The subgroup $C=C_G(T)$ is normal in $G$. If $C\ne1$, it contains a
minimal normal subgroup $M$ of $G$. But then $M\leq T\cap C=Z(T)=1$, a
contradiction. Hence $C_G(T)=1$.

Each $S_i$ is subnormal in $G$ and satisfies $S_i^{\N}=S_i$: a
nontrivial nilpotent quotient of a nonabelian simple group is
impossible. Therefore $\W(G)\leq N_G(S_i)$ for every $i$, and
$\W(G)\leq K$. Theorem~\ref{thm:functor} also gives
$\W(G)\leq\pi^{-1}(\W(Q))$. This proves one inclusion
in~\eqref{eq:radical-free}.

For the converse, take $x\in K\cap\pi^{-1}(\W(Q))$, and let $H\sn G$
be arbitrary. Use the notation $A,B,L$ of Lemma~\ref{lem:support}.
The subgroup $\pi(H)=HT/T$ is subnormal in $Q$, and
Lemma~\ref{lem:residual-basics}(ii) yields
\[
 (HT/T)^{\N}=LT/T.
\]
Since $\pi(x)\in \W(Q)$, the element $x$ normalizes $LT$. Since
$x\in K$, it normalizes every $S_i$ and hence normalizes $B$ and
$C_G(B)$. Formula~\eqref{eq:reconstruct} now shows that $x$ normalizes
$L=H^{\N}$. This holds for every subnormal $H$, so $x\in \W(G)$,
proving~\eqref{eq:radical-free}.

Every element of $T$ belongs to $K$ and maps to $1$ under $\pi$.
Thus~\eqref{eq:radical-free} gives $T\leq \W(G)$; taking quotients
gives~\eqref{eq:modsoc}, and with $T\le\W(G)\le K$ the index formula.

Finally, conjugation embeds $K$ into $\prod_i\Aut(S_i)$ because its
kernel is $C_G(T)=1$. Compose this map with the natural maps to the
outer automorphism groups. Its kernel consists exactly of $T$. Indeed,
an element inducing inner automorphisms on all factors has the same
action on $T$ as some element of $T$, and their quotient centralizes
$T$, hence is trivial. This proves~\eqref{eq:outer}.
\end{proof}

\begin{corollary}\label{cor:radical-free-fixed}
Under the hypotheses of Theorem~\ref{thm:radical-free}:
\begin{enumerate}
\renewcommand{\labelenumi}{\textup{(\roman{enumi})}}
\item $\W(G)=K$ if and only if $D\le\W(Q)$;
\item $\W(G)=G$ if and only if every $S_i$ is normal in $G$ and
$\W(G/T)=G/T$; in this case $G/T$ is soluble;
\item if every $\Out(S_i)$ is nilpotent, then $\W(G)=K$, and
$G/\W(G)$ is isomorphic to the permutation group induced by $G$ on
$\{S_1,\ldots,S_t\}$.
\end{enumerate}
\end{corollary}
\begin{proof}
(i) follows from~\eqref{eq:modsoc}. By~\eqref{eq:radical-free},
$\W(G)=G$ is equivalent to $K=G$ and $\W(Q)=Q$, and $K=G$ says
exactly that every $S_i$ is normal in $G$. In that case $Q=D$ is
soluble by~\eqref{eq:outer} and the Schreier conjecture. In (iii),
$D$ is nilpotent by~\eqref{eq:outer} and normal in $Q$, so
$D\le\Fit(Q)\le\W(Q)$ by Lemma~\ref{lem:Fstar}, and (i) applies.
\end{proof}

\begin{remark}
The hypothesis of (iii) holds when every $S_i$ is an alternating
group, a sporadic group or a group $\mathrm{PSL}_2(q)$, whose outer
automorphism groups are abelian \cite{ATLAS,Wilson,Lyons}. In
particular, every almost simple group with such a socle is a
$\W$-group, and for $T=A_5^t\le G\le S_5\wr S_t$ one has
$\W(G)=G\cap S_5^t$. Although $\Rad(G)=1$, the quotient $Q=G/\Soc(G)$
may have a nontrivial soluble radical, so
Theorem~\ref{thm:radical-free} cannot in general be iterated.
\end{remark}

\subsection{Nonsoluble length}
The \emph{nonsoluble length} $\lambda(G)$ is the number of
nonsoluble factors in a shortest normal series each of whose factors
is soluble or a direct product of nonabelian simple groups. We use
\[
 \lambda(G/N)\le\lambda(G),\qquad
 \lambda(G)\le\lambda(N)+\lambda(G/N)
 \quad(N\trianglelefteq G),
\]
see \cite[Section~2]{KS} and \cite[Section~2]{DS}.

\begin{theorem}\label{thm:fixed-structure}
Suppose $\W(X)=X$, and let $L/\Rad(X)=\Soc(X/\Rad(X))$. Then $L/\Rad(X)$
is a direct product of nonabelian simple groups, each of them normal in
$X/\Rad(X)$, and $X/L$ embeds in the direct product of their outer
automorphism groups. Consequently $X/L$ is soluble and
$\lambda(X)\le1$.
\end{theorem}
\begin{proof}
Let $Y=X/\Rad(X)$, a $\W$-group by Theorem~\ref{thm:functor}. If
$Y=1$ there is nothing to prove. Otherwise $\Rad(Y)=1$, and
Theorem~\ref{thm:radical-free} and
Corollary~\ref{cor:radical-free-fixed}(ii) give the assertions about
$\Soc(Y)$ and $Y/\Soc(Y)$; the latter is soluble by the Schreier
conjecture. The series $1\le\Rad(X)\le L\le X$ shows that
$\lambda(X)\le1$.
\end{proof}

\begin{corollary}\label{cor:lambda}
For every group $G$, $\lambda(\W(G))\le1$ and
$\lambda(G)\le\ellW(G)$.
\end{corollary}
\begin{proof}
By Theorem~\ref{thm:structure}(iv), each layer
$\W_{i+1}(G)/\W_i(G)=\W(G/\W_i(G))$ is a $\W$-group, so it has
nonsoluble length at most one by Theorem~\ref{thm:fixed-structure}.
Subadditivity along the upper series proves the assertion.
\end{proof}

\begin{remark}
The index $|G:\W(G)|$ does not control $\ellW(G)$ in either direction
beyond trivial bounds. On the one hand, for $G\ne1$,
$\ellW(G)=1+\ellW(G/\W(G))$ is at most one plus the number of prime
factors of $|G:\W(G)|$. On the other hand, if $S$ is nonabelian
simple and $G=S\wr C_n$ with $n\ge2$, then every coordinate subgroup
$S_i$ is subnormal with $S_i^{\N}=S_i$, so
$\W(G)\le\bigcap_iN_G(S_i)=S^n=\Fit^*(G)\le\W(G)$. Hence
$|G:\W(G)|=n$ while $\ellW(G)=2$.
\end{remark}

\section*{Acknowledgements}
This work was supported by the National Natural Science Foundation of
China (grants 12371021 and 12101135). The third author was supported by
grant PID2024\allowbreak-159495NB\allowbreak-I00 from the Ministerio de
Ciencia, Innovaci\'{o}n y Universidades, Spain, and by grant
CIAICO/\allowbreak2023/\allowbreak007 from the Conselleria
d'Educaci\'{o}, Universitats i Ocupaci\'{o}, Generalitat Valenciana.
During this work, the third author also held a Specially Appointed
Professor position at Nantong University.

\section*{AI Declaration}
During the preparation of this work the authors used the AI assistants
Claude (Anthropic) and Codex (OpenAI) to develop and check mathematical
arguments, to propose examples, to write and run GAP programs for
exploratory computations, to check bibliographic references, and to edit
the text. The authors take full responsibility for the
mathematical content.

\end{document}